\documentclass[twoside,12pt,reqno]{amsart}
\usepackage{amscd, bbm, amsaddr}
\usepackage{amssymb,mathabx}
\usepackage{todonotes}
\usepackage{mathrsfs,wasysym, tikz, extpfeil}
\usetikzlibrary{matrix, arrows}
\usepackage{amsmath}
\usepackage[margin=2cm]{geometry}

\usepackage[pdftex,colorlinks,hypertexnames=false,linkcolor=black,urlcolor=black,citecolor=black]{hyperref}

\makeatletter
\@addtoreset{equation}{section}

\numberwithin{equation}{section}

\newtheorem{Theorem}{Theorem}[section]
\newtheorem*{Theorem*}{Theorem}
\newtheorem*{Corollary*}{Corollary}
\newtheorem*{Problem}{Problem}
\newtheorem{Lemma}[Theorem]{Lemma}
\newtheorem{Proposition}[Theorem]{Proposition}
\newtheorem{Corollary}[Theorem]{Corollary}
\newtheorem*{MainTheorem}{First Theorem}
\newtheorem*{SecondTheorem}{Second Theorem}
\newtheorem*{ThirdTheorem}{Third Theorem}

\theoremstyle{definition}

\theoremstyle{remark}
\newtheorem{Remark}[Theorem]{Remark}
\newtheorem*{Remark*}{Remark}
\newtheorem*{proofofPBW}{Proof of Theorem~\ref{T:PBWenvalg}}

\newbox\squ  
\setbox\squ=\hbox{\vrule width.6pt
   \vbox{\hrule height.6pt width.4em\kern1ex
         \hrule height.6pt}%
   \vrule width.6pt}

\newcommand{\C}{\mathbb{C}}
\newcommand{\N}{\mathbb{N}}
\newcommand{\Z}{\mathbb{Z}}
\renewcommand{\k}{\mathbbm{k}}

\renewcommand{\Z}{\mathcal{Z}}

\newcommand{\g}{\mathfrak{g}}

\renewcommand{\sl}{\mathfrak{sl}}

\newcommand{\m}{\mathfrak{m}}
\newcommand{\n}{\mathfrak{n}}

\newcommand{\p}{\mathfrak{p}}
\newcommand{\q}{\mathfrak{q}}

\renewcommand{\O}{\mathcal{O}}
\newcommand{\U}{\mathcal{U}}

\newcommand{\I}{\mathcal{I}}

\newcommand{\V}{\mathcal{V}}

\newcommand{\Co}{\mathcal{C}}
\renewcommand{\P}{{\mathcal{P}}}
\renewcommand{\L}{{\mathcal{L}}}

\newcommand{\F}{\mathcal{F}}

\newcommand{\ad}{\operatorname{ad}}

\newcommand{\Ht}{\operatorname{ht}}
\newcommand{\Lie}{\operatorname{Lie}}
\newcommand{\Ker}{\operatorname{Ker}}

\newcommand{\SL}{\operatorname{SL}}
\newcommand{\gr}{\operatorname{gr}}

\newcommand{\Spec}{\operatorname{Spec}}

\newcommand{\End}{\operatorname{End}}

\newcommand{\Mat}{\operatorname{Mat}}

\newcommand{\Ind}{\operatorname{Ind}}

\newcommand{\Der}{\operatorname{Der}}

\newcommand{\PSpec}{\operatorname{\P-Spec}}
\newcommand{\PPrim}{\operatorname{\P-Prim}}
\newcommand{\Pmod}{\operatorname{-\P-mod}}
\newcommand{\PMod}{\operatorname{-\P-Mod}}
\renewcommand{\mod}{\operatorname{-mod}}
\newcommand{\Mod}{\operatorname{-Mod}}
\newcommand{\Ann}{\operatorname{Ann}}

\newcommand{\Prim}{\operatorname{Prim}}
\newcommand{\Ass}{\operatorname{Ass}}
\newcommand{\opp}{{\operatorname{opp}}}

\newcommand{\oT}{\overline{T}}

\newcommand{\odelta}{\overline{\delta}}

\title[]{\boldmath The orbit method for Poisson orders}

\author{Stephane Launois$^{\dagger}$ and Lewis Topley$^{\ddagger}$}
\address{\begin{center}\small{School of Mathematics, Statistics and Actuarial Science, \\ University of Kent, Canterbury\\ United Kingdom CT2 7FS} \end{center}}

\email{$\dagger$: \tt S.Launois@kent.ac.uk}
\email{$\ddagger$: \tt L.Topley@kent.ac.uk }


\begin{document}

\begin{abstract}
A version of Kirillov's orbit method states that the primitive spectrum of a generic quantisation $A$ of a Poisson algebra $Z$ should correspond bijectively to the symplectic leaves of $\Spec(Z)$. In this article we consider a Poisson order $A$ over a complex affine Poisson algebra $Z$. We stratify the primitive spectrum $\Prim(A)$ into symplectic cores, which should be thought of as families of non-commutative symplectic leaves. We then introduce a category $A\PMod$ of $A$-modules adapted to the Poisson structure on $Z$, and we show that when $\Spec(Z)$ is smooth with locally closed symplectic leaves, there is a natural homeomorphism from the spectrum of annihilators of simple objects in $A\PMod$ to the set of symplectic cores in $\Prim(A)$ with its quotient topology. Several application are given to Poisson representation theory.
Our main tool is the Poisson enveloping algebra $A^e$ of a Poisson order $A$, which captures the Poisson representation theory of $A$. For $Z$ regular and affine we prove a PBW theorem for $A^e$ and use this to characterise the annihilators of simple Poisson modules: they coincide with the Poisson weakly locally closed,
the Poisson primitive and the Poisson rational ideals. We view this as a generalised weak Poisson Dixmier--M{\oe}glin equivalence.
\end{abstract}

\maketitle
\section{Introduction}
\subsection{The orbit method}
\label{S:orbmth}
Kirillov's orbit method appears in a wide variety of contexts in representation theory and Lie theory, and is occasionally referred to as a philosophy rather than
a theory, on account of the fact that it serves as a guiding principle in many cases where it cannot be formulated as a precise statement.
The original manifestation of the orbit method states that characters of simple
modules for Lie groups can be expressed as normalisations of Fourier transforms of certain functions on coadjoint orbits \cite{Ki}, but perhaps the most concrete
algebraic expressions of the orbit method is a well-known theorem of Dixmier which asserts that when $G$ is a complex solvable algebraic group and $\g = \Lie(G)$
the primitive ideals of the enveloping algebra $U(\g)$ lie in natural one-to-one correspondence with the set-theoretic coadjoint orbit space $\g^\ast/G$; see \cite[Theorem~6.5.12]{Di}.
Dixmier's theorem fails for complex simple Lie algebras \cite[Remark~9.2(c)]{Go1}, however progress has been made recently by
Losev \cite{Lo} using techniques from deformation theory to show that $\g^*/G$ canonically maps to $\Prim U(\g)$, and the map is an embedding in classical types.
The image consists of a certain completely prime ideals and conjecturally it is always injective.

The Kirillov--Kostant--Souriau theorem asserts that the coadjoint orbits are actually the symplectic leaves of the Poisson variety $\g^*$, and so a broad
interpretation of the orbit method philosophy is the following: suppose that $Z$ is a Poisson algebra and $A$ is a quantisation of $Z$,
then the primitive spectrum $\Prim(A)$ should correspond closely to the set of symplectic leaves of $\Spec(Z)$, indeed a slightly more general principle
was suggested by Goodearl in \cite[\textsection 4.4]{Go1}. There are several examples
of quantum groups and quantum algebras where the \emph{correspondence} we allude to here actually manifests itself as a \emph{bijection}, or better
yet a homeomorphism, once the set of leaves is endowed with a suitable topology; the reader should refer to \cite{Go1} where numerous correspondences
of this type are surveyed.

\subsection{Poisson orders and their modules}

In deformation theory, Poisson algebras arise as the semi-classical limits of quantisations, as we briefly recall.
If $A$ is a  torsion-free $\C[q]$-algebra where $q$ is a parameter, such that $A_0 := A/ q A$ is commutative then $A_0$ is equipped with a Poisson bracket by setting
\begin{eqnarray}
\label{e:hiyachi}
\{\pi(a_1), \pi(a_2)\} := \pi (q^{-1}[a_1, a_2])
\end{eqnarray}
where $\pi : A \rightarrow A_0$ is the natural projection and $a_1, a_2 \in A$. Of course we do not need to assume that
$A_0$ is commutative to obtain a Poisson algebra since formula \eqref{e:hiyachi} endows the centre $Z(A_0)$ with the structure of a Poisson algebra regardless.
In fact, something stronger is true: by choosing $\pi(a_1) \in Z(A_0)$ and $\pi(a_2) \in A_0$ formula \eqref{e:hiyachi} endows $A_0$ with a
biderivation 
\begin{eqnarray}
\label{e:Pbrack}
\{ \cdot, \cdot\} : Z(A_0) \times A_0 \rightarrow A_0
\end{eqnarray}
which restricts to a Poisson bracket $\{ \cdot, \cdot\} : Z(A_0) \times Z(A_0) \rightarrow Z(A_0)$.
In \cite{BGo} Brown and Gordon axiomatised this structure in cases where $A_0$ is a $Z(A_0)$-module of finite type by saying that
\emph{$A_0$ is a Poisson order over $Z(A_0)$}. The precise definition
will be recalled in \textsection\ref{S:poissonorders}, and a slightly more general approach to constructing
Poisson orders in deformation theory will be explained in \textsection \ref{S:POexamples}.
The bracket \eqref{e:Pbrack} induces a map $H : Z(A_0) \rightarrow \Der_\C(A_0)$ and the image is referred to as the set of \emph{Hamiltonian derivations of $A_0$}.
In \emph{op. cit.} they proved some very attractive general results with the ultimate goal of better understanding
the representation theory of symplectic reflection algebras. In this paper we pursue the themes of the orbit method
in the abstract setting of Poisson orders.

When $Z$ is a Poisson algebra and $A$ is a Poisson order over $Z$ we define a Poisson $A$-module to be an $A$-module with a compatible
action for the Hamiltonian derivations $H(Z)$; see \textsection\ref{S:poissonorders}. In the case where $A = Z$
these modules are closely related to $\mathcal{D}$-modules over the affine variety $\Spec(Z)$
(see Remark~\ref{R:Dmods}), and they have appeared in the literature many times (see \cite{Ar, Fa, GP, LOV, Oh2} for example).
In the setting of Poisson orders a similar category of modules was studied in \cite{Ti}.
For a Poisson $A$-module $M$ we define \emph{the singular support of $M$} to be the subset
\begin{eqnarray}
\label{e:SS}
\V(M) := \{ I \in \Prim(A) \mid \Ann_A(M) \subseteq I\}.
\end{eqnarray}
The set of annihilators of simple Poisson $A$-modules will always be equipped with its Jacobson topology. 

\subsection{Symplectic cores vs. annihilators of simple Poisson modules}
\label{S:sympcoresandmodules}
A primitive ideal of $A$ is the annihilator of a simple $A$-module and the set of such ideals equipped with their Jacobson topology is called the primitive
spectrum, denoted $\Prim(A)$. It is often then case that simple $A$-modules cannot be classified but $\Prim(A)$ can be described completely,
which offers good motivation for studying primitive spectra.
The \emph{Poisson core} of an ideal $I\subseteq A$ is the largest ideal $\P(I)$ of $A$ contained in $I$ which is stable under the Hamiltonian derivations,
and we define an equivalence relation on the set $\Prim(A)$ by saying $I \sim J$ if $\P(I) = \P(J)$. The equivalence
classes are called the \emph{symplectic cores of $\Prim(A)$}, the set of symplectic cores is denoted $\Prim_\Co(A)$ and the symplectic core of $\Prim(A)$
containing $I$ is denoted $\Co(I)$. We view $\Prim_\Co(A)$ as topological space endowed with the quotient topology.
In case $Z$ is an affine Poisson algebra such that $\Spec(Z)$ has Zariski locally closed symplectic leaves, \cite[Proposition~3.6]{BGo}
shows that the symplectic leaves coincide with the symplectic cores of $\Spec(Z)$; see also Proposition~\ref{P:algebraicleaves}. Thus the cores of $\Prim(A)$ can occasionally be regarded as non-commutative
analogues of symplectic leaves.

\begin{MainTheorem}
Suppose that $\Spec(Z)$ is a smooth complex affine Poisson variety with Zariski locally closed symplectic leaves, and $A$ is a Poisson order
over $Z$. For every simple Poisson $A$-module $M$, there is a unique symplectic core of $\Prim(A)$ which is dense in $\V(M)$ in the Jacobson topology.
Sending $\Ann_A(M)$ to this symplectic core sets up a homeomorphism between the space of annihilators of simple Poisson $A$-modules and the space $\Prim_\Co(A).$
\end{MainTheorem}
The hypothesis that the symplectic leaves of $\Spec(Z)$ are algebraic can be replaced with something strictly weaker; see Remark~\ref{R:leafalg}.
Our theorem has obvious parallels with Joseph's irreducibility theorem \cite{Jo}, which states that for $\g$ a complex semisimple Lie algebra
and $M$ a simple $\g$-module the variety $\{\chi \in \g^* \mid \chi(\gr \Ann_{U(\g)}(M)) = 0\}$ contains a unique dense nilpotent orbit.
Other closely related results can be found in \cite{Gi, Ma}, although all of the papers we cite here apply to Poisson structures which have finitely many symplectic leaves -
this is a hypothesis we do not require. It is natural to wonder whether our first theorem might serve as a starting point for a new proof of the irreducibility theorem.

In order to illustrate in what sense our theorem is an expression of the orbit method philosophy we record the following special case
where $A =Z = \C[\g^*]$ is the natural Poisson structure arising on the dual of an algebraic Lie algebra.
\begin{Corollary*}
Let $G$ be a connected complex linear algebraic group and $\g = \Lie(G)$ its Lie algebra. Then the set of 
annihilators of simple Poisson $\C[\g^*]$-modules lies in natural bijection with the set $\g^*/G$ of coadjoint orbits. 
\end{Corollary*}
\begin{Remark*}
When $G$ is a complex solvable algebraic group and $\g = \Lie(G)$ we may apply Dixmier's theorem (Cf. \textsection \ref{S:orbmth}) and the previous corollary
to deduce that the annihilators of simple Poisson $\C[\g^*]$-modules lie in one-to-one correspondence with annihilators of simple $U(\g)$-modules.
In \textsection \ref{S:solvablePLA} we will make a detailed study of these parameterisations and show that they are dual to each other in a
precise sense.
\end{Remark*}

Another interesting class of Poisson algebras to which the theorem can be applied are the classical finite $W$-algebras.
Let $G$ be a connected complex reductive algebraic group, $e \in \g = \Lie(G)$ a nilpotent element and $(e,h,f)$ and $\sl_2$-triple. If $\g_f$ denotes the centraliser of $f$ in $\g$,
then the slice $e + \g_f$ is transversal to the orbit $G\cdot e$ and it inherits a natural Poisson structure from $\g \cong \g^*$ via Hamiltonian reduction; see \cite{Gi1} for more detail.
\begin{Corollary*}
The annihilators of simple Poisson $\C[e+\g_f]$-modules lie in natural one-to-one correspondence with the connected components of the sets $(G\cdot x) \cap (e + \g_f)$ where $x \in e + \g_f$.
\end{Corollary*}
\begin{proof}
Thanks to \cite[Theorem~3.1]{Va} we know that the symplectic leaves of $e + \g_f$ are the connected components of the sets $(G\cdot x) \cap (e + \g_f), x\in\g$, in particular they are Zariski locally closed.
\end{proof}
We now briefly describe the proof of the first theorem.
Showing that the bijection is a homeomorphism is easy, and only requires us to make a comparison between the topologies of the two spaces involved,
which is carried out in \textsection \ref{S:bijishom}. The hard work lies in proving that the support of a simple module contains a dense core,
and that this sets up a bijection. Conceptually this is achieved in three steps. Firstly we we show that whenever the leaves of $\Spec(Z)$
are locally closed, the closures of the symplectic cores of $\Prim(A)$ are determined by the \emph{Poisson primitive ideals}, which we define to be $\PPrim(A) := \{\P(I) \mid I \in \Prim(A)\}$.
Secondly we show that $\PPrim(A)$ is equal to the set of annihilators of simple Poisson $A$-modules.
Finally, we observe in Lemma~\ref{L:keyspeclemma} that there is a unique symplectic core in each closure $\overline{\Co(I)}$, and it follows that there exists a bijection
$ \PPrim(A) \leftrightarrow \Prim_\Co(A)$. The first and second steps are really consequences of our second main theorem,
which gives a detailed comparison of different types of $H(Z)$-stable ideals in Poisson orders, as we now explain.

\subsection{The Poisson Dixmier--M{\oe}glin equivalence for Poisson orders}
\label{S:thePDMEsec}
Let $\k$ be any field, $Z$ be an affine Poisson $\k$-algebra and $A$ a Poisson order over $Z$.
As usual $\Spec(A)$ denotes the set of prime two-sided ideals, and we have $\Prim(A) \subseteq \Spec(A)$.
The \emph{Poisson ideals of $A$} are the two-sided ideals which are stable under the Hamiltonian derivations $H(Z)$.
We write $\PSpec(A)$ for the set of all Poisson ideals of $A$ which are also prime.
The set $\PSpec(A)$ endowed with its Jacobson topology is referred to as the \emph{Poisson spectrum of $A$},
and the elements are known as \emph{Poisson prime ideals of $A$}.
If $A$ is prime and $Z$ is an integral domain then we can form the field of fractions $Q(Z)$, and since $A$ is a $Z$-module of finite type
the tensor product $A \otimes_Z Q(Z)$ is isomorphic to $Q(A)$ the division ring of fractions of $A$; in particular $Q(A)$ exists.
When $I \in \Spec(A)$ we have $I \cap Z \in \Spec(Z)$ (see \cite[Theorem~10.2.4]{MR}, for example) and so we can form the division ring $Q(A/I)$.
If $I$ is a Poisson ideal then the set of derivations $H(Z)$ acts naturally on $A/I$, and the action extends to an action on $Q(A/I)$ by the Leibniz rule
\begin{eqnarray}\label{leibniz}
\delta(ab^{-1}) = \delta(a)b^{-1} + ab^{-1}\delta(b) b^{-1},
\end{eqnarray}
where $\delta \in H(Z)$ and $a,b \in A/I$ with $b \neq 0$. The centre of $Q(A/I)$ will be written $C Q(A/I)$ and
we define the \emph{Poisson centre of $Q(A/I)$} to be the subalgebra
\begin{eqnarray}
\label{e:Pcentre}
C_{\P} Q(A/I):= \{z \in C Q(A/I) \mid H(Z)z = 0\}.
\end{eqnarray}
Let $I \in \PSpec(A)$ be a Poisson prime ideal. We say that:
\begin{itemize}
\item $I$ is \emph{Poisson locally closed} if $\{I\}$ is a locally closed subset of $\PSpec(A)$;

\item $I$ is \emph{Poisson weakly locally closed} if there are only finitely many Poisson prime ideals in $A/I$ of height one;

\item $I$ is \emph{Poisson primitive} if $I = \P(J)$ for some $J \in \Prim(I)$;

\item $I$ is \emph{Poisson rational} if $C_{\P} Q(A/I)$ is a finite extension of $\k$;

\item $I$ is \emph{the annihilator of a simple Poisson module} if $$I = \{a \in A \mid a M = 0\}$$ for some simple Poisson $A$-module $M$.
\end{itemize}

\begin{SecondTheorem}
Let $A$ be a Poisson order over an affine Poisson $\C$-algebra $Z$. The following hold:
\begin{enumerate}
\item[(a)] every Poisson locally closed ideal in $\PSpec(A)$ is Poisson primitive;
\item[(b)] every annihilator of a simple Poisson $A$-module is is Poisson rational;
\item[(c)] when $Z$ is regular, every Poisson primitive ideal of $A$ is the annihilator of some simple Poisson $A$-module.
\end{enumerate}
Furthermore the following families of ideals in $\PSpec(A)$ coincide:
\begin{enumerate}
\item[(i)] Poisson weakly locally closed ideals;
\item[(ii)] Poisson primitive ideals;
\item[(iii)] Poisson rational ideals.
\end{enumerate}
Finally, the following are equivalent:
\begin{itemize}
\item[(I)] the Poisson primitive ideals of $A$ are Poisson locally closed on $\PSpec(A)$;
\item[(II)] the symplectic cores of $\Prim(A)$ are locally closed sets cut out by the Poisson primitive ideals
\end{itemize}
When the symplectic leaves of $Z$ are locally closed in the Zariski topology conditions (I) and (II) hold for $Z$ as a Poisson order over itself.
If these equivalent conditions hold for $Z$ then they also hold for $A$.
\end{SecondTheorem}

The equivalence of (i), (ii), (iii) is known as \emph{the weak Poisson Dixmier--M{\oe}glin equivalence}. It has recently been proven for
Poisson algebras by the first named author et al. in \cite{BLLM}, and our approach is to lift the theorem to the setting of Poisson orders
using the close relationships between prime and primitive ideals in finite centralising extensions. Part (a) is similarly a well-known fact
in the setting of Poisson algebras, and the same proof works here. When the Poisson primitive ideals of a Poisson algebra $Z$ 
are Poisson locally closed we say that {\em the Poisson Dixmier--M{\oe}glin equivalence} (PDME) holds for $Z$. Generalising this rubric,
we shall say that the PDME holds for $A$ when conditions (I) and (II) hold for $A$.
It was an open question from \cite{BGo} as to whether every affine Poisson algebra satisfies the PDME, however
recently counterexamples have been discovered \cite{BLLM}. We may now rephrase the last sentence of the second theorem:
we have shown that if the PDME holds for a complex affine Poisson algebra $Z$ then it holds for every Poisson order $A$ over $Z$.

\subsection{The universal enveloping algebra of a Poisson order}
To conclude our statement of results it remains to offer some commentary on (b) and (c) of the second main theorem.
Statement (c) was proposed in the case $A = Z$ in \cite{Oh2} although the proof contains an error
\footnote{We thank Professor Oh for his clarifications and assistance here, and for allowing us to reproduce the example in Remark~\ref{R:hypothesisremark}.}. The converse was conjectured at the same time
and proven in case $Z$ is a polynomial algebra in \cite{OPS}. Since our results are stated in the setting of Poisson modules over Poisson orders, we require new tools and 
new methods. Our main technique is to define and study the {\em (universal) enveloping algebra of a Poisson order}.
This is an associative algebra $A^e$ generated by symbols $\{m(a), \delta(z) \mid a\in A, z\in Z\}$ subject to certain relations \eqref{alghom}-\eqref{deltaxy}
such that category $A^e\Mod$ of left modules is equivalent to the category of Poisson $A$-modules.
Using this construction we are able to define localisation of Poisson modules over Poisson orders, which is our main tool in proving part (b) of the second main theorem.
In order to prove part (c) we show that
when $Z$ is regular, $A^e$ is a free (hence faithfully flat) $A$-module (Corollary~\ref{C:faithfully}), which implies that the ideals of $A^e$ are closely related to the ideals
of $A$ (Cf. Lemma~\ref{idealslemma}). The fact that $A^e$ is $A$-free follows quickly from our last main theorem of this paper, which we view as a PBW theorem
for the enveloping algebras of Poisson orders. There is a natural filtration $A^e = \bigcup_{i \geq 0} \F_i A^e$, defined by
placing generators $\{m(a) \mid a\in A\}$ in degree 0 and $\{\delta(z) \mid z\in Z\}$ in degree 1, which we
call \emph{the PBW filtration of $A^e$}. The associated graded algebra is denoted $\gr A^e$.
The statement and proof of our third and final main theorem are quite similar to Rinehart's PBW theorem for Lie algebroids \cite{Ri}.
\begin{ThirdTheorem}
Suppose that $Z$ is affine and regular over a field. Then the natural surjection $$A \otimes_Z S_Z(\Omega_{Z/\k}) \twoheadrightarrow \gr A^e$$
induced by multiplication in $\gr A^e$ is an isomorphism. 
\end{ThirdTheorem}

\subsection{Structure of the paper} We now describe the
structure of the current paper. In \textsection\ref{S:prelims} we state the definition of a Poisson order: our definition is very slightly different to
the one originally given in \cite{BGo}, although a careful comparison is provided in Remark~\ref{R:comparison}. Poisson modules over Poisson orders
are also defined, and in \textsection\ref{S:POexamples} we go on to to give examples of elementary constructions of Poisson orders and their modules, as well as reviewing the
common construction of examples in deformation theory. In \textsection\ref{S:sympcores} we study the symplectic cores of $\Prim(A)$, and
prove the equivalence of (I) and (II), as well as the subsequent two assertions of second main theorem.
In \textsection\ref{S:PBWandlocalisation} we introduce the enveloping algebra of a Poisson order.
We state the universal property in \textsection\ref{S:envalgfirst} and prove a criterion for $A^e$ to be noetherian.
In \textsection\ref{S:localisation} we use $A^e$ to define and study localisations of Poisson $A$-modules, whilst in \textsection\ref{S:PBWproof}
we prove the PBW theorem and state some useful consequences. In \textsection\ref{S:primitiveandsimples} we prove (b) and (c) of the second theorem
using the tools developed in of \textsection\ref{S:PBWandlocalisation}. In \textsection \ref{S:weakPDME} we prove (a) and the equivalence of (i), (ii), (iii) in 
the second main theorem. Following \cite{Go} we observe that results, such as the PDME, can be studied in the slightly more general context of
finitely generated algebras equipped with a set of distinguished derivations, and it is in this setting that we prove the results of \textsection \ref{S:weakPDME}.
Finally, in \textsection\ref{S:proofofthm1} we show that the second theorem implies the first. In \textsection \ref{S:solvablePLA} we make a careful comparison between Dixmier's
bijection $\g^* / G \rightarrow \Prim U(\g)$ and our bijection $\{\text{annihilators of simple Poisson }\C[\g^*]\text{-modules}\} \rightarrow \g^* / G $
in the case where $\g$ is a solvable, and finally in \textsection \ref{S:appoftheorem} we discuss some famous examples
of Poisson orders arising in deformation theory to which our first main theorem can be applied. We conclude the article by posing some questions
about their Poisson representation theory.

\subsection{A discussion of related results and new directions}
\label{S:relres}
It is worth mentioning that our first main theorem is very close in spirit to a conjecture of Hodges and Levasseur \cite[\textsection 2.8, Conjecture 1]{HL} which seeks to relate the
primitive spectrum of the quantised coordinate ring of a complex simple algebra group $\O_q(G)$ in the case where $q$ is a generic parameter,
to the Poisson spectrum of the classical limit $\O(G)$; see \cite[\textsection 4.4]{Go1} for a survey of results.
Although the spectra are always known to lie in natural bijection, this bijection is only known to be a homeomorphism in case $G = \SL_2(\C)$ and $\SL_3(\C)$ \cite{Fr}.
By contrast, our bijection is always a homeomorphism, however our results only apply to these families of algebras when the parameter is a root of unity.
It would be natural to attempt to strengthen this comparison.

Although our results are fairly comprehensive we expect that part (c) of the second main theorem should hold without the hypothesis that $Z$ is regular,
and so the first main theorem should hold true without assuming $\Spec(Z)$ is smooth. Note that the symplectic leaves of a singular Poisson variety can be defined, thanks to
\cite[\textsection 3.5]{BGo}. This would constitute an extremely worthwhile development, as there are important
examples of Poisson orders over singular Poisson varieties, eg. rational Cherednik algebras. At least this should be achievable for Poisson orders over
isolated surface singularities using the methods of \cite[\textsection 3.4]{LOV} along with our proof of Theorem~\ref{primisann}, which only depends upon the PBW theorem for $A^e$.


Another motivation for this work is the following: there appear to be deep connections
between the dimensions of simple modules of a Poisson order $A$ over $Z$ and the dimensions of its symplectic leaves of $Z$. We expect that the Poisson representation theory of $A$
will be closely related to the representation theory of $A$, and so the current paper will lay the groundwork for such relationships to be understood in a broader context.

\medskip
\noindent {\bf Ackowledgements:} The authors would like to thanks Ken Brown, Francesco Esposito, Baohua Fu, Iain Gordon, David Jordan and Sei-Qwon Oh for useful discussions and correspondence
on the subject of this paper. We would also like to thank the referee for making useful comments and for suggesting freeness, and its proof, in Corollary~\ref{C:faithfully}(iv).
This research was funded in part by EPSRC grant EP/N034449/1.

\section{Preliminaries}
\label{S:prelims}
\subsection{Notations and conventions}

For the first and second sections we let $\k$ be any field whilst in subsequent sections we shall work over $\C$.
When the ground field is fixed all vector spaces, algebras and unadorned tensor products will be defined over
this choice of field.

When we say that $A$ is a algebra, we mean a not necessarily commutative unital $\k$-algebra.
When we say that $A$ is \emph{affine} we mean that it is semiprime and finitely generated.
By an $A$-module we mean a left module, unless otherwise stated. By a primitive ideal we always mean the annihilator of a left $A$-module.
The category of all $A$-modules is denoted $A\Mod$
and the subcategory of finitely generated $A$-modules is denoted $A\mod$.

When we say that $A$ is \emph{filtered} we mean
that there is a non-negative $\mathbb{Z}$-filtration $\F_0 A \subseteq \F_1 A \subseteq \F_2 A \subseteq \cdots \subsetneq A$
satisfying $\F_{-1} A = 0$, $\F_0 A = \k$, $A = \bigcup_{i \geq 0} \F_i A$ and $(\F_i A) (\F_j A) \subseteq \F_{i+j} A$ for all $i, j \geq 0$.
As usual the associated graded algebra of a filtered algebra is $\gr A := \bigoplus_{i \geq 0} \F_i A / \F_{i-1} A$, and if $a \in \F_i A$ then we
refer to the image of $a$ in $\F_i A / \F_{i-1} A \subseteq \gr A$ as \emph{the top graded component of $a$}. Furthermore, $A$ is said to be
\emph{almost commutative} if $\gr A$ is commutative.
\subsection{Poisson orders and their modules}
\label{S:poissonorders}
A \emph{Poisson algebra} is a commutative algebra $Z$ endowed with a skewsymmetric $\k$-bilinear biderivation $\{\cdot, \cdot\} : Z \times Z \rightarrow Z$
which makes $Z$ into a Lie algebra. 
Let $A$ be a $Z$-algebra which is a module of finite type over $Z$. 
We say that $A$ is a \emph{Poisson order (over $Z$)} if the Poisson bracket on $Z$ extends to a map
\begin{eqnarray}
\label{e:biderbracket}
\{\cdot ,\cdot\} : Z \times A \longrightarrow A.
\end{eqnarray}
such that the map $H(x) := \{x, \cdot\} : A \rightarrow A$ satisfies the following for all $x, y\in Z$ and $a\in A$:
\begin{enumerate}
\item[(i)] $H(x) \in \Der_\k(A)$;
\item[(ii)] $H(xy)a = xH(y)a + yH(x)a$.
\end{enumerate}
In other words, $\{\cdot, \cdot\}$ is a biderivation. Note that $H(x)y = - H(y)x$ and $H(x)Z \subseteq Z$ for all $x, y\in Z$, since $Z$ is a Poisson algebra
and \eqref{e:biderbracket} extends the Poisson bracket on $Z$.
When the choice of Poisson algebra $Z$ is clear we usually refer to $A$ simply as a Poisson order.
We refer to $H(Z)$ as the set of \emph{Hamiltonian derivations of $A$}. In most cases of interest $A$ will be a faithful $Z$-module
and in these cases we view $Z$ as an $H(Z)$-stable
subalgebra of $A$ via the algebra homomorphism $Z \rightarrow A; z\mapsto z1_A$.
\begin{Remark}
\label{R:comparison}
The definition of a Poisson order in \cite{BGo} is slightly weaker than the one given here, as they only assume property (ii)
of $H$ in the case where $x,y,a \in Z$. Our justification for choosing this definition is twofold: firstly the most interesting examples which arise in deformation
theory satisfy these slightly stronger properties; see \textsection \ref{S:POexamples}. Secondly the stronger definition
suggests a stronger definition for a Poisson $A$-module, and the enveloping algebra for this category of modules satisfies
the PBW theorem of \textsection\ref{S:PBWandlocalisation}, which is fundamental to all of our results. 
\end{Remark}

When $Z$ is a fixed Poisson algebra and $A$ is a Poisson order over $Z$, we define a \emph{Poisson $A$-module}
to be an $A$-module $M$ together with a linear map 
$$\nabla : Z \rightarrow \End_\k(M)$$
such that for all $x,y \in Z$, all $a \in A$ and all $m \in M$ we have
\begin{itemize}
\item[(i)] $\nabla(xy) m = x\nabla(y)m + y \nabla(x) m$;
\smallskip
\item[(ii)] $\{x,a \} m = [\nabla(x), a] m$;
\smallskip
\item[(iii)] $\nabla(\{x,y\})m = [\nabla(x), \nabla(y)]m$.
\end{itemize}
The morphisms of Poisson $A$-modules are defined in the obvious manner, and the category of all Poisson $A$-modules will be denoted $A\PMod$.
Since Poisson $A$-modules are Poisson $Z$-modules by restriction, we are considering a special class of flat Poisson connections.
\begin{Remark}
\label{R:simplesnotfg}
It is not true that simple Poisson $A$-modules are necessarily finitely generated over $A$. For example when $A = Z = \C[\g^*]$ and
$\g$ is a simple Lie algebra, it is not hard to see that simple Poisson $A$-modules annihilated by the augmentation ideal $(\g) \unlhd A$
are the same as simple $\g$-modules, and these are often infinite dimensional. We thank Ben Webster for this useful observation.
\end{Remark}

\subsection{Examples of Poisson orders and their modules}
\label{S:POexamples}
Every Poisson algebra is a Poisson order over itself. Furthermore, for $Z$ fixed there are several constructions which allow us to construct
new Poisson orders over $Z$ from old ones. Let $A$ be a Poisson order over $Z$. Then:
\begin{enumerate}
\item $\Mat_n(A)$ is a Poisson order for any $n > 0$, with $$H(z) (a_{i,j})_{1\leq i,j\leq n} := (H(z) a_{i,j})_{1\leq i,j\leq n}.$$
\item the opposite algebra $A^\opp$ is a Poisson order;
\item the tensor product $A \otimes_Z B$ of two Poisson orders is again a Poisson order with $$H(z) (a\otimes b) := \{z,a\} \otimes b + a \otimes \{z, b\}.$$
\item when $M$ is a Poisson $A$-module the endomorphism ring $\End_Z(M)$ is a Poisson order and the subalgebra $\End_A(M)$ is a sub-Poisson order,
with $$(H(z) e)m := [\nabla(z), e] m$$
for any $z \in Z$, $m\in M$ and $e\in \End_Z(M)$.
\end{enumerate}
The above constructions are very suggestive of a theory of a Brauer group over $Z$ adapted to the theory of Poisson orders. This is a theme we
hope to pursue in future work.

All other examples of Poisson orders which we will be interested in arise in the context of deformation theory. We follow \cite{KR} closely.
Let $R$ be a commutative associative algebra and let $(\epsilon) \subseteq R$ be a principal prime ideal, and write $k = R/ (\epsilon)$. Consider the $k$-algebra $A_\epsilon := A/ \epsilon A$
with centre $Z_\epsilon$ and write $N_\epsilon$ for the preimage of $Z_\epsilon$ in $A$ under the natural projection $\pi : A \twoheadrightarrow A_\epsilon$. For $a\in N_\epsilon$ and $b \in A$
we have $[a,b] \in \epsilon A$ and so we may define
\begin{eqnarray}
\label{e:poissbrackdef}
\{\pi(a), \pi(b)\} := \pi (\epsilon^{-1} [a,b])
\end{eqnarray}
generalising \eqref{e:hiyachi}. When $A_\epsilon$ is finite over $Z_\epsilon$ the bracket \eqref{e:poissbrackdef} makes $Z_\epsilon$ into a Poisson algebra. If $Z \subseteq Z_\epsilon$
is any Poisson subalgebra such that $A_\epsilon$ is a $Z$-module of finite type then $A_\epsilon$ becomes a Poisson order over $Z$. Notable examples include:
\begin{enumerate}
\item when $R = \mathbb{Z}$, $\epsilon = p \in \mathbb{Z}$ is any prime number and $\g_\mathbb{Z}$ is the Lie algebra of a $\mathbb{Z}$-group scheme then $U(\g_\mathbb{Z}) / pU(\g_\mathbb{Z}) \cong U(\g_p)$
where $\g_p := \g_\mathbb{Z} \otimes_\mathbb{Z} \mathbb{F}_p$ and $\mathbb{F}_p := \mathbb{Z}/(p)$. It is well known that $\g_p$ is a restricted Lie algebra and the calculations of \cite{KR} show that
the $p$-centre $Z_p(\g)$ is a Poisson subalgebra of the centre of $U(\g_p)$, naturally isomorphic to $\mathbb{F}_p[\g_p^*]$ with its Lie--Poisson structure.
Since $U(\g_p)$ is finite over the $p$-centre we see that $U(\g_p)$ is a Poisson order over $Z_p(\g)$.
\item let $R = \C[t^{\pm 1}]$, $\epsilon = (t - q_0)$ for some primitive $\ell$th root of unity $q_0 \in \C$, and $A$ is any of the following: a quantised enveloping algebra
of a complex semisimple Lie algebra, a quantised coordinate ring of a complex algebraic group, any quantum affine space. It is well known that the
$\ell$th powers of the standard generators of $A_\epsilon$ generate a central subalgebra $Z_0$ over which $A_\epsilon$ is a finite module, and \eqref{e:poissbrackdef}
equips $Z_0$ with the structure of a complex affine Poisson algebra, so $A_\epsilon$ is a Poisson order over $Z_0$.
\end{enumerate}

We continue with $A$ a Poisson order over $Z$ and we list some elementary examples of Poisson modules.
The first example of a Poisson $A$-module is $A$, with map $\nabla$ defined by $\nabla(z)a := \{z,a\}$. If $I \unlhd A$ is any left ideal
which is also Poisson, then both $I$ and the quotient $A/I$ admit the structure of a Poisson $A$-module. One obvious
source of such ideals are those of the form $AI$ where $I \unlhd Z$ is any Poisson ideal. A method for constructing Poisson $A$-modules from Poisson
$Z$-modules occurs as a special case of the following crucial lemma.
\begin{Lemma}
\label{L:tensormodulelemma}
Suppose that $B \subseteq A$ are Poisson orders over a Poisson algebra $Z$ and that $M$ is a Poisson $B$-module with structure map $\nabla_B$. Then ${A\otimes_{B}M}$
is naturally an $A$-module, and is additionally a Poisson $A$-module with structure map $\nabla_A$ defined by $\nabla_A(x) (a\otimes m) := \{x,a\} \otimes m + a \otimes \nabla_B(x)m$.
\end{Lemma}
\begin{proof}
To see that $\nabla_A$ is well-defined we must check that the kernel of the natural map $A\otimes_\k M \twoheadrightarrow A\otimes_B M$ is preserved
by $\nabla_A(Z)$. Let $x \in Z$, $a \in A$, $b\in B$, $m \in M$ and write $a\otimes m := a \otimes_\k m$. We see that $\nabla_A(x) (ab\otimes m - a\otimes bm)$ is equal to
\begin{eqnarray*}
& & \{x, ab\} \otimes m + ab \otimes \nabla_B(x) m - \{x,a\} \otimes bm - a \otimes \nabla_B(bm) \\
& = & (\{x,a\} b \otimes m - \{x,a\} \otimes bm) + (a \{x, b\} \otimes m - a \otimes \{x,b\}m) \\ & & + (ab \otimes \nabla_B(x)m  - a \otimes b \nabla_B(x) m).
\end{eqnarray*}
This confirms that $\nabla_A(Z)$ is well-defined on $A \otimes_B M$. For the rest of the proof tensor products $a \otimes m$ will be taken over $B$.
The first axiom of a Poisson $A$-module follows from the calculation
\begin{eqnarray*}
\nabla_A(xy) (a\otimes m) &=& \{xy, a\} \otimes m + a \otimes \nabla_B(xy) m \\ &=& (\{x,a\}y + x\{y,a\}) \otimes m +  a\otimes (x\nabla_B(y)m + y\nabla_B(x)m)\\
&=& x(\{y,a\} \otimes m + a\otimes \nabla_B(x) m) + y(\{x,a\} \otimes m + a\otimes \nabla_B(x) m) \\
&=& x \nabla_A(y)(a \otimes m) + y \nabla_A(x) (a \otimes m)
\end{eqnarray*}
where $x,y \in Z$, $a \in A$ and $m \in M$. The second axiom of a Poisson module is a consequence of the next calculation, in which $a,b\in A$ and $x, m$ are as before
\begin{eqnarray*}
[\nabla_A(x) , a] (b\otimes m) &=& \nabla_A(x)(ab \otimes m) - a (\{x,b\} \otimes m + b \otimes \nabla_B(x) m)\\
&=& \{x,ab\} \otimes m + ab \otimes \nabla_B(x) m - a\{x,b\}\otimes m - ab \otimes \nabla_B(x)m \\
&=& \{x,a\} b \otimes m = \{x,a\}(b\otimes m).
\end{eqnarray*}
The third axiom of a Poisson module only regards the Lie algebra structure and so follows from the Hopf algebra structure
on the universal enveloping algebra of the Lie algebra $Z$, since $A$ and $M$ are Poisson $Z$-modules.
\end{proof}

\begin{Remark}
\label{R:Dmods}
It was observed in \cite[Proposition~1.1]{Fa} that when $Z$ is a symplectic affine Poisson algebra over $\C$ every Poisson $Z$-module arises from a unique $\mathcal{D}$-module
on $\Spec(Z)$.
\end{Remark}

\subsection{Symplectic cores in primitive spectra}
\label{S:sympcores}

We continue with an affine Poisson $\k$-algebra $Z$ and a Poisson order $A$ over $Z$. If $\mathcal{S}$ is any collection of ideals of $A$ then we can endow
$\mathcal{S}$ with the \emph{Jacobson topology} by declaring the sets $\{I \in \mathcal{S} \mid \bigcap_{J \in S} J \subseteq I \}$ to be closed, where $S \subseteq \mathcal{S}$
is any subset. We will refer to such a set $\mathcal{S}$ as \emph{a space of ideals} to suggest that we are equipping it with the Jacobson topology.
The space of prime ideals and primitive ideals of
$A$ are denoted $\Spec(A)$ and $\Prim(A)$ respectively. A ring is \emph{Jacobson} if every prime ideal is an intersection of primitive ideals;
clearly this property is equivalent to the statement that $\Prim(A)$ is a topological subspace of $\Spec(A)$, not just a subset.
It is well known that $Z$ is Jacobson, since it is affine and commutative, and so it follows from \cite[9.1.3]{MR} that $A$ is a Jacobson ring.

The $H(Z)$-stable ideals of $Z$ and $A$ are called \emph{Poisson ideals} and the space of prime Poisson ideals 
is called the Poisson spectrum, denoted $\PSpec(Z)$ and $\PSpec(A)$ respectively. Recall that for any ideal $I \subseteq A$ the Poisson core
$\P(I)$ is the largest Poisson ideal contained in $I$; by \cite[3.3.2]{Di} we have $\P(I)$ prime whenever $I$ is prime, and the same holds for $Z$.
\begin{Lemma}
\label{L:Pandcap}
Let $A$ be a Poisson order over $Z$ and $I \subseteq J \subseteq A$ are any ideals with $I$ Poisson.
Denote the quotient map $\pi : A \rightarrow A/I$. We have:
\begin{enumerate}
\item[(i)] $\P(J) \cap Z = \P(J \cap Z)$;
\item[(ii)] $\pi \P(J) = \P(\pi J)$. 
\end{enumerate}
\end{Lemma}
\begin{proof}
To prove (i) it suffices to observe that $\P(J \cap Z) \subseteq \P(J) \cap Z \subseteq J \cap Z$, by the definition of $\P$ whilst (ii) follows from the fact that
$\pi$ defines an inclusion preserving bijection between the set of Poisson ideals of $A/ I$ and the set of Poisson ideals of $A$ which contain $I$.
\end{proof}

\begin{Remark}
It is not hard to see that the topology on $\PSpec(A)$ is the subspace topology from the embedding $\PSpec(A) \subseteq \Spec(A)$: if $I \subseteq A$
is any ideal then the set $\{J \in \PSpec(A) \mid I \subseteq J\}$ is equal to $\{J \in \PSpec(A) \mid \bigcap_{i \in S} I_i \subseteq J\}$ where $\{I_i \mid i \in S\}$
is the set of minimal Poisson ideals over $I$.
\end{Remark}

Our purpose now is to define the symplectic stratification of the primitive spectrum $\Prim(A)$ used in the statement of the first main theorem.
Consider the following diagram, which is commutative by part (i) of Lemma~\ref{L:Pandcap}:
\begin{eqnarray}
\begin{array}{c}
\begin{tikzpicture}[node distance=2cm, auto]
\pgfmathsetmacro{\shift}{0.3ex}
\node (A) {$\Prim(A)$};
\node (E) [right of=A] {$ $};
\node (B) [right of=E] {$\PSpec(A)$};

\node (C) [below of=A] {$\Prim(Z)$};
\node (F) [right of=C] {$ $};
\node (D) [right of= F] {$\PSpec(Z)$};

\draw[->] (A) --(B) node[above,midway] {$\P$};
\draw[->] (C) --(D) node[below,midway] {$\P$};
\draw[->] (A) --(C) node[left,midway] {$ $};
\draw[->] (B) --(D) node[right,midway] {$ $};

\end{tikzpicture}
\end{array}
\end{eqnarray}
The vertical arrows denote contraction of ideals $I \mapsto I \cap Z$. The fibres of the map $\Prim(Z) \rightarrow \PSpec(Z)$
are called \emph{the symplectic cores of $\Prim(Z)$}, and they were first studied by Brown and Gordon in \cite{BGo}. We define
\emph{the symplectic cores of $\Prim(A)$} to be the fibres of the map $\Prim(A) \rightarrow \PSpec(A)$. 
For $\m \in \Prim(Z)$ we write $\Co(\m)$ for the symplectic core of $\m$ and for $I\in \Prim(A)$ we write $\Co(I)$ for the symplectic core of $I$.
The following result shows that the symplectic cores of $\Prim(Z)$ are closely related to the symplectic leaves; the first part was proven in \cite[Proposition~3.6]{BGo},
and the second statement in \cite[Theorem~7.1(c)]{Go1}.
\begin{Proposition}
\label{P:algebraicleaves}
Let $Z$ be a complex affine Poisson algebra. For $\m \in \Prim(Z)$, write $\L(\m)$ for the symplectic leaf of $\m$. We have
$\L(\m) \subseteq \Co(\m)$ with equality if the symplectic leaves of $\Prim(Z)$ are algebraic. More generally
$$\Co(\m) = \overline{\L(\m)} \setminus \bigcup \L(\n)$$
where the union is taken over all $\n \in \Prim(Z)$ such that $\L(\n) \subsetneq \overline{\Co(\m)}$. 
\end{Proposition}
Thus we think of the symplectic cores of $\Prim(A)$ as being something similar to the symplectic leaves of the primitive spectrum.
If the Poisson primitive ideals of $A$ are Poisson locally closed then we say that \emph{the Poisson Dixmier--M{\oe}glin equivalence (PDME) holds for $A$}.
Later on in the paper (Lemma~\ref{L:locimpprimimprat}) we will show that Poisson locally closed ideals are always Poisson primitive.
The following result is an extension of \cite[Lemma~3.3]{BGo}.
\begin{Lemma}
\label{L:PDME}
Let $A$ be a Poisson order over a complex affine Poisson algebra $Z$. Write $N_I := \bigcap \P(K)$ where the intersection is taken over all $K \in \Prim(A)$ with $\P(I) \subsetneq \P(K)$.
The following are equivalent:
\begin{enumerate}
\item[(i)] the PDME holds for $A$;
\item[(ii)] the symplectic cores of $\Prim(A)$ are locally closed subsets defined by the Poisson cores.
In other words, for all $I \in \Prim(A)$ we have
\begin{eqnarray}
\label{e:PDME1}
\Co(I) = \{ J \in \Prim(A) \mid \P(I) \subseteq J \text{ and } N_I \nsubset J\};
\end{eqnarray}

\item[(iii)] the inclusion $\P(I) \subseteq N_I$ is proper.
\end{enumerate}
Furthermore, if the PDME holds for $Z$ as a Poisson order over itself, then it also holds for $A$.
\end{Lemma}
\begin{proof}
If $I \in \Prim(A)$ then, using Lemma~\ref{intover}, $\{\P(I)\}$ is a locally closed subset of $\PPrim(A)$ if and only if
the intersection $N_I$ properly contains $\P(I)$, so (i) $\Leftrightarrow$ (ii). We point out that the lemma just cited
does not depend on any of the results of this article which precede it, and follows straight from the definitions.
It is not hard to see that for $I \in \Prim(A)$ we have
$$\Co(I) = V(\P(I)) \setminus V(N_{I})$$
if and only if $N_I \neq \P(I)$, from which the equivalence of (ii) and (iii) follows.

Now suppose that $\P(I\cap Z) \subsetneq \bigcap \P(\m)$ where the intersection is taken over all ideals $\m \in \Prim(Z)$
such that $\P(I \cap Z) \subsetneq \P(\m)$. Using Lemma~\ref{L:Pandcap},
we deduce that $\P(I) \cap Z = \P(I\cap Z) \subsetneq \bigcap \P(\m)$, where the intersection is taken
over all $\m \in \Prim(Z)$ such that $I \cap Z \subsetneq \P(\m)$. By the incomparability property over essential extensions \cite[Theorem~6.3.8]{FS} 
we see that  $\P(I) \subsetneq \P(J)$ implies $\P(I \cap Z) \subsetneq \P(J \cap Z)$ and so from $\bigcap \P(\m) \subseteq N_I \cap Z$ we deduce that
$\P(I) \subsetneq N_I$. We conclude from (iii) $\Rightarrow$ (i) that the PDME holds for $A$.

\end{proof}

\section{The universal enveloping algebra of a Poisson order}
\label{S:PBWandlocalisation}
Throughout this entire section we work over an arbitrary field $\k$.
Let $Z$ be a Poisson $\k$-algebra and let $A$ be a Poisson order over $Z$.

\subsection{Definition and first properties of the enveloping algebra}
\label{S:envalgfirst}
Poisson $A$-modules can be thought of as modules
over a non-associative algebra due to the action of the derivations $\nabla(Z)$, and one encounters elementary
technical problems with dealing with such modules. For example, if $M$ is a simple Poisson $A$-module
then it is not necessarily finitely generated over $A$ (Cf. Remark~\ref{R:simplesnotfg}); this contrasts with the situation for simple
$A$-modules where any such module is generated by any nonzero element. To remedy this problem we take a viewpoint
which is common in universal algebra: we write down an associative algebra
whose module category is equivalent to $A\PMod$ and we use this new algebra to study simple Poisson $A$-modules and their annihilators.

The \emph{Poisson enveloping algebra $A^e$} of the Poisson order $A$ over $Z$ is the $\k$-algebra with generators
\begin{eqnarray}
\label{e:Aegens}
\{\alpha(a) \mid a \in A\} \cup \{\delta(z) \mid z \in Z\}
\end{eqnarray}
and relations 
\begin{eqnarray}
\label{alghom} & & \alpha : A \rightarrow A^e \text{ is a unital algebra homomorphism};\\
\label{liehom} & & \delta : Z \rightarrow A^e \text{ is a Lie algebra homomorphism};\\
\label{mdelta} & & \alpha(\{x, a\}) = [\delta(x), \alpha(a)];\\
\label{deltaxy} & & \delta(xy) = \alpha(x) \delta(y) + \alpha(y) \delta(x),
\end{eqnarray}
for all $x,y \in Z$ and all $a\in A$. Recall that the Poisson algebra $Z$ is a Poisson order over itself and we write $Z^e$ for the enveloping algebra
of $Z$. The algebra $Z^e$ has been extensively studied in the mathematical literature, although the first results appeared in \cite{Ri}, since
Poisson algebras are examples of Lie--Rinehart algebras. Our next observation follows straight from the relations.
\begin{Lemma}
\label{L:psilemma}
There is a natural homomorphism $Z^e \rightarrow A^e$ which sends the elements $\{\alpha(z), \delta(z) \mid z\in Z\}$ of $Z^e$
to the elements of $A^e$ with the same names.
\end{Lemma}
Next we record some criteria for $A^e$ to satisfy the ascending chain condition on ideals.
\begin{Lemma}
\label{L:noetherian}
If $Z$ is noetherian or $A$ is finitely generated then both $A$ and $A^e$ are noetherian.
\end{Lemma}
\begin{proof}
The Artin--Tate lemma shows that when $A$ is finitely generated so too is $Z$, and so $Z$ is noetherian. It suffices to
prove that when $Z$ is noetherian so too are $A$ and $A^e$. 
The extension $Z\subseteq A$ is centralizing in the sense of \cite[10.1.3]{MR} and so Corollary~10.1.11 of that book
shows that $A$ is noetherian. Now by the relations the map of rings $\alpha : A \rightarrow A^e$ is almost normalizing in the sense
of \cite[1.6.10]{MR} and so the lemma follows from Theorem~1.6.14 of the same book.
\end{proof}

When $a\in A^e$ we write $\ad(a)$ for the derivation of $A^e$ given by $b \mapsto ab - ba$.
The following two statements can be proven by induction on \eqref{mdelta} and \eqref{deltaxy} respectively.
\begin{Lemma}
\label{L:extendedrels}
For $x_1,...,x_n \in Z$ and $a \in A$ we have:
\begin{enumerate}
\item[(i)] $\alpha(H(x_1) \cdots H(x_n) a) = \ad(\delta(x_1))\cdots \ad(\delta(x_n)) a$;
\item[(ii)] $\delta(x_1 \cdots x_n) = \sum_{i=1}^n x_1 \cdots x_{i-1} \hat x_i x_{i+1} \cdots x_n \delta(x_i)$.
\end{enumerate}
\end{Lemma}

We define a filtration on $A^e$
by placing $A$ in degree $0$ and $\delta(z)$ in degree $1$ for all $z \in Z$.
We call the resulting filtration \emph{the PBW filtration on $A^e$}
\begin{eqnarray}
\label{e:PBWfilt}
A^e = \bigcup_{i\geq 0} \F_i A^e
\end{eqnarray}
and as usual we denote the associated graded algebra by $\gr A^e$. One of our main tools in this paper is a precise description of $\gr A^e$, which we give in
Theorem~\ref{T:PBWenvalg}. For now we record a precursor to that result which will be needed when describing localisations of Poisson modules.
\begin{Lemma}
\label{L:envgenset}
If $\{a_i \mid i \in I\}$ generates $A$ and $\{x_i \mid i \in J\}$ generates $Z$ as $\k$-algebras, where $I$ and $J$ are index sets and $J$ is totally ordered,
then $A^e$ is spanned by the monomials
\begin{eqnarray*}
\alpha(a_{i_1}) \cdots \alpha(a_{i_n}) \delta(x_{j_1}) \cdots \delta(x_{j_m})
\end{eqnarray*}
where $i_1,...,i_n \in I$ and $j_1 \leq \cdots \leq j_m$ lie in $J$. The same statement holds with the elements $\alpha(a)$ occurring after the elements $\delta(x)$.
\end{Lemma}
\begin{proof}
It follows from relations \eqref{alghom} and part (ii) of Lemma~\ref{L:extendedrels} that the algebra $A^e$ is generated by the set
$\{\alpha(a_i) , \delta(x_j) \mid i \in I, j\in J\}$. 
Therefore the lemma will follow from the claim that $\gr Z^e$ is central in $\gr A^e$.
This is clear upon examining the top graded components of relations \eqref{liehom} and \eqref{mdelta}. 
\end{proof}

We now record the universal property of $A^e$ which allows us to view Poisson $A$-modules as $A^e$-modules.
Consider the category $\U$ whose objects are triples $(B, \alpha', \delta')$ where $B$ is an associative algebra
with unital algebra homomorphism $\alpha' : A \rightarrow B$ and Lie algebra homomorphism $\delta' : Z \rightarrow B$
satisfying \eqref{mdelta} and \eqref{deltaxy}, and where the morphisms
$(B, \alpha', \delta') \rightarrow (C, \alpha'', \delta'')$ between two objects in $\U$ are the algebra homomorphisms
$\beta : B\rightarrow C$ making the diagrams below commute.
$$
\begin{tikzpicture}[node distance=1.5cm, auto]
\pgfmathsetmacro{\shift}{0.3ex}
\node (A) {$A$};
\node (B) [right of=A] {$B$};
\node (C) [below of=B] {$C$};

\node (D) [right of= B] {$Z$};
\node (E) [right of=D] {$B$};
\node (F) [below of=E] {$C$};

\draw[->] (A) --(B) node[above,midway] {$\alpha'$};
\draw[->] (A) --(C) node[left,midway] {$\alpha''$};
\draw[->] (B) --(C) node[right,midway] {$\beta$};

\draw[->] (D) --(E) node[above,midway] {$\delta'$};
\draw[->] (D) --(F) node[left,midway] {$\delta''$};
\draw[->] (E) --(F) node[right,midway] {$\beta$};

\end{tikzpicture}
$$
\begin{Lemma}
\label{L:categoryequivalence}
Let $A$ be a Poisson $Z$-order.
\begin{enumerate}
\item{$(A^e, \alpha, \delta)$ is an initial object in the category $\U$;}
\item{There is a category equivalence
$$A\PMod \cong A^e\Mod.$$
If $(M, \nabla)$ is a Poisson $A$-module then it becomes a $A^e$-module by defining
\begin{eqnarray}
\label{e:alph} \alpha(a) m := am;\\
\label{e:delt} \delta(x) m := \nabla(x) m,
\end{eqnarray}
for $x \in Z$, $a\in A$ and $m \in M$. Conversely if $M$ is a $A^e$-module then \eqref{e:alph}, \eqref{e:delt} make $M$ in to a Poisson $A$-module.
Consequently the $A^e$-module homomorphisms are precisely the Poisson $A$-module homomorphisms.}
\end{enumerate}
\end{Lemma}
\begin{proof}
Part (1) is an immediate consequence of the definition of a Poisson $A$-module, whilst part (2) follows directly from part (1).
\end{proof}
\begin{Remark}
We may now define the category $A\Pmod$, of \emph{finitely generated Poisson $A$-modules}, to be the essential image of $A^e\mod$ in $A\PMod$ under the above equivalence.
\end{Remark}

\begin{Corollary}
\label{C:alphainj}
The map $\alpha : A \rightarrow A^e$ is injective.
\end{Corollary}
\begin{proof}
Recall from Example~\ref{S:POexamples} that $A$ is a Poisson $A$-module, hence an $A^e$-module. This induces a homomorphism $\rho : A^e \rightarrow \End_\k(A)$.
Now it suffices to notice that the map $\rho_1 : A^e \rightarrow A$ given by $\rho_1(x) := \rho(x) 1_A$ is a left inverse to $\alpha$.
\end{proof}
\begin{Remark}
\label{R:relsremark}
Thanks to the lemma we may (and shall) identify $A$ with a subalgebra of $A^e$. Thus we view $A^e$ as the algebra generated by $A$ and $\delta(Z)$
subject to relations \eqref{liehom}, \eqref{mdelta}, \eqref{deltaxy} with every instance of $\alpha(a)$ replaced by $a$.
\end{Remark}

\subsection{Localisation of Poisson $A$-modules}
\label{S:localisation}
It is well known that if $S\subseteq Z$ is any multiplicative subset containing no zero divisors then the localisation $Z_S := Z[S^{-1}]$
carries a unique Poisson algebra structure such that the natural map $Z \rightarrow Z_S$ is a Poisson algebra homomorphism.
Briefly, this structure is defined by extending the Hamiltonian derivations to $Z_S$ via \eqref{leibniz}. The reader
may refer to the proof of \cite[Lemma~1.3]{Ka} for the precise formula.
In the same manner, when a multiplicative set $S\subseteq Z$ consists of non-zero divisors of $A$ the algebra $A_S := A \otimes_Z Z_S$
carries a unique structure of a $Z_S$-Poisson order and $A_S$ is a faithful $Z_S$-module.
Let $A_S^e$ denote the  Poisson enveloping algebra of $A_S$.
\begin{Lemma} 
\label{L:localenvalg}
If $S \subseteq Z$ is any multiplicative subset then the natural map $Z_S \otimes_Z A^e \rightarrow A_S^e$ induced by multiplication is surjective.
\end{Lemma}
\begin{proof}
By relation \eqref{deltaxy}, for $a\in A$ and $s\in S$ we have
$\delta(a) = \delta(as^{-1} s) = as^{-1}\delta(s) + s\delta(as^{-1})$
which can be rewritten as $\delta(as^{-1}) = s^{-1}\delta(a) - as^{-2}\delta(s).$
Applying Lemma~\ref{L:envgenset}, this shows that the desired map surjects.
\end{proof}

From the generators and relations of $A^e$ we see there is a natural map $A^e \rightarrow A_S^e$ sending generators of $A^e$
to the elements in $A^e_S$ with the same name.
Now if $M$ is any Poisson $A$-module and $S \subseteq Z$ is a multiplicative subset then we can define the localisation $M_S$
by viewing $M$ as an $A^e$-module, and then defining
\begin{eqnarray}
M_S := A^e_S \otimes_{A^e} M.
\end{eqnarray}
This is an $A^e_S$-module and thus it is a Poisson $A_S$-module, via the equivalence described in Lemma~\ref{L:categoryequivalence}, (2).
The \emph{torsion subset of $M$} is defined by
\begin{eqnarray}
T(M) := \{z \in Z\mid zm = 0 \text{ for some } m \in M\}.
\end{eqnarray}
\begin{Proposition}
\label{P:simplepoisson}
If $M$ is a simple Poisson $A$-module then $M_S$ is either zero or a simple Poisson $A_S$-module.
Furthermore $M_S$ is zero if and only if $$S \cap T(M) \neq \emptyset.$$
\end{Proposition}
\begin{proof}
Since $M$ is a Poisson $A$-submodule of $M_S$ it follows that the kernel of $M \rightarrow M_S$ is a Poisson $A$-submodule.
Since $M$ is simple we have $M_S \neq 0$ if and only if $\Ker(M \rightarrow M_S) = 0$, and it follows from \cite[Ch. II, \textsection~2, No. 4, Proposition~10(ii)]{Bo} the natural map
$M \rightarrow M_S$ is injective if and only if $S\cap T(M) = \emptyset$.
Suppose that $S\cap T(M) = \emptyset$ so that $M_S \neq 0$. Then according to Lemma~\ref{L:localenvalg} we have a surjection
$Z_S \otimes_Z A^e \otimes_{A^e} M \twoheadrightarrow M_S$. In other words $M_S$ is spanned by expressions $s^{-1} m$ with $s \in S$ and $m \in M$.
The following calculation shows that every element of $M_S$
can be written in the form $s^{-1}m$: for $s_1,...,s_n \in S$ and $m_1,...,m_s \in M$ we have
$$\sum_{i=1}^n s_i^{-1} m_i = (\prod_{i=1}^n s_i)^{-1} \sum_{i=1}^n s_1 \cdots \hat s_i \cdots s_n m_i.$$
 Now we use the following characterisation of simple $A^e$-modules: they are precisely the modules which are generated by
any non-zero element. Pick $0 \neq s^{-1}m \in M_S$ and let $t^{-1} n$ be any other element. Since $M$ is a simple $A^e$-module there
is $a \in A^e$ such that $am = n$, and it follows that $t^{-1}as(s^{-1}m) = t^{-1}n$. Hence $M_S$ is generated as an $A_S^e$-module
by any nonzero element, and so $M_S$ is a simple Poisson $A_S$-module as required.
\end{proof}

\begin{Remark}
When $\p \in \Spec(Z)$ we adopt the usual convention of writing $A_\p$ and $Z_\p$ for the localisations $A_{S\setminus \p}$ and $Z_{S\setminus \p}$.
When $z \in Z \setminus \{0\}$ is not nilpotent we write $A_z$ and $Z_z$ for the localisations at the multiplicative set $\{z^i \mid i \geq 0\}$.
\end{Remark}

\subsection{A Poincar\'{e}--Birkhoff--Witt theorem for the enveloping algebra}
\label{S:PBWproof}

Our present goal is to describe the associated graded algebra $\gr A^e$ with respect to the PBW filtration \eqref{e:PBWfilt}
in the case where $Z$ is a regular Poisson algebra, i.e. when $\Spec(Z)$ is a smooth affine variety.
Let $\Omega := \Omega_{Z/\k}$ denote the $Z$-module of K\"{a}hler differentials for $Z$; see \cite[Ch. II, \textsection 8]{Ha} for an overview.
The relations in the enveloping algebra imply that there is a natural map $A \otimes_Z S_Z(\Omega) \rightarrow \gr A^e$ which is surjective.
The PBW theorem for Poisson orders takes the following form.
\begin{Theorem}
\label{T:PBWenvalg}
Suppose that $Z$ is a regular, affine Poisson algebra over an algebraically closed field $\k$. The following hold:
\begin{enumerate}
\item{The natural surjective algebra homomorphism
$$A\otimes_Z S_Z(\Omega) \xtwoheadrightarrow{} \gr A^e$$
is an isomorphism;}
\item{There is an isomorphism of $(A, Z^e)$-bimodules $$A \otimes_Z Z^e \overset{\sim}{\longrightarrow} A^e;$$}
\item{There is an isomorphism of $(Z^e, A)$-bimodules $$Z^e\otimes_Z A \overset{\sim}{\longrightarrow}A^e.$$}
\end{enumerate}
\end{Theorem}

The proof will occupy the rest of the current subsection. The approach is modelled on that of \cite{Ri} where a similar result was proven for
enveloping algebras of Lie--Rinehardt algebras.
We first prove the theorem in the case where $\Omega$ is a finitely generated free $Z$-module and then use localisation of Poisson orders
to deduce the theorem in the case where $\Omega$ is locally free, i.e. projective. By \cite[Theorem~8.15]{Ha} we know that $\Omega$ is a projective $Z$-module
if and only if $Z$ is regular, from which we will conclude the theorem.

Suppose that $\Omega$ is a free $Z$-module of finite type, so there exist $z_1,...,z_n \in Z$ such that $d(z_1),...,d(z_n)$ is a basis for $\Omega$.
Therefore the symmetric algebra $S_Z(\Omega)$ is free over $Z$ and the ordered products $d(z_I) := d(z_{i_1}) \cdots d(z_{i_m})$
with $i_1\leq \cdots \leq i_m$ provide a basis. When $I$ is a sequence $1\leq i_1 \leq \cdots \leq i_m \leq n$ we write $|I| = m$ and write
$j \leq I$ if $j \leq i_1$.
\begin{Lemma}\label{Haminsym}
Let $\Omega$ be a free $Z$-module with a finite basis. There is a Poisson $A$-module structure on $A\otimes_Z S_Z(\Omega)$ such that 
\begin{eqnarray}\label{delz}
\delta(z_j) (1\otimes d(z_I)) = 1\otimes d(z_j)d(z_I)
\end{eqnarray}
whenever $j \leq I$.
\end{Lemma}
\begin{proof}
It was first observed by Huebschmann that when $Z$ is a Poisson algebra $\Omega$ carries a natural Lie algebra structure
and that $(Z, \Omega)$ is a Lie--Rinehardt algebra; see \cite[Theorem~3.11]{Hue}. Therefore we may apply the first part of
the proof of \cite[Theorem~3.1]{Ri} to deduce that $S_Z(\Omega)$ carries a Poisson $Z$-module structure satisfying \eqref{delz}
provided $\Omega$ is free. Using Lemma~\ref{L:tensormodulelemma} we see that $A \otimes_Z S_Z(\Omega)$ carries
the required Poisson $A$-module structure.

\end{proof}

\begin{proofofPBW}
We start by proving the statement of part (1) of the Theorem, however for the moment we replace the hypothesis that $Z$ is regular with
the assumption that $\Omega$ is a free $Z$-module of finite rank. We adopt the notation introduced preceding Lemma~\ref{Haminsym}
so that $z_1,...,z_n$ is a basis for $\Omega$ over $Z$, and we write
$$\odelta(z_I) := \delta(z_1)^{i_1} \cdots \delta(z_n)^{i_n} + \F_{\sum i_j - 1} A^e \in \gr A^e.$$
We have $A = \F_0 A^e \cong \F_0 A^e / \F_{-1} A^e \subseteq \gr A^e$ and so $\gr A^e$ is a left $A$-module.
We need to show that the set
\begin{eqnarray}
\label{e:argo}
\{\odelta(z_I) \mid I \in \mathbb{Z}_{\geq 0}^n\}
\end{eqnarray}
spans a free left $A$-submodule of $\gr A^e$.
Observe that the Poisson $A^e$-module structure defined in Lemma~\ref{Haminsym}
makes $T := A \otimes_Z S_Z(\Omega)$ into a filtered $A^e$-module, and so $\gr (T) \cong T$
is a graded $\gr A^e$-module. We denote the operation $\gr A^e \otimes_\k T \rightarrow T$ by $u \otimes a \mapsto u \cdot a$.
Thanks to \eqref{delz} the map $\psi : \gr A^e \rightarrow T$ defined by $u\mapsto u\cdot (1\otimes 1)$ sends $\odelta(z_I)$
to $1\otimes d(z_1)^{i_1} \cdots d(z_n)^{i_n}$ for $I = (i_1,...,i_n)$. Since $\psi$ is $A$-equivariant and the image of \eqref{e:argo} is 
$A$-linearly independent we deduce that \eqref{e:argo} is $A$-linearly independent, as claimed. This proves part (1) in the case where
$\Omega$ is a free $Z$-module.

Now we suppose that $Z$ is regular. Then it follows from \cite[Theorem 8.15]{Ha} that $\Omega$ is a locally free
$Z$-module in the sense that there is a function $r : \Spec Z \rightarrow \N_0$ such that
$$\Omega_{Z_\p/\k} \cong Z_\p \otimes_Z \Omega_{Z/\k}\cong Z^{r(\p)}_\p$$ as $Z$-modules, for all $\p \in \Spec Z$.
By the previous paragraph we deduce that the natural map $A_\p \otimes_{Z_\p} S_{Z_\p}(\Omega_{Z_\p/\k}) \rightarrow \gr (A_\p^e)$
is an isomorphism. This shows that there is a commutative diagram of algebra homomorphisms:
\begin{center}
\begin{tikzpicture}[node distance=1.8cm, auto]
\node (A) {$A \otimes_Z S_Z(\Omega)$};
\begin{scope}[node distance=6cm and 10cm]
\node (B) [right of=A] {$\gr A^e$};
\end{scope}
\node (C) [below of=A] {$\prod_{\p \in \Spec Z}(A_\p \otimes_{Z_\p} S_{Z_\p}(\Omega_{Z_\p/\k}))$};
\begin{scope}[node distance=6cm and 10cm]
\node (D) [right of=C] {$\prod_{\p\in \Spec Z}\gr (A_\p^e)$};
\end{scope}
 \draw[>=stealth', ->>] (A) to node {$ $} (B);
 \draw[left hook-latex] (A) to node {$ $} (C);
 \draw[>=stealth', ->] (C) to node {$\sim$} (D);
 \draw[>=stealth', ->] (B) to node {$ $} (D);
\end{tikzpicture}
\end{center}
We point out that the natural map $A \otimes_Z S_Z(\Omega) \to \prod_{\p \in \Spec Z} A_\p \otimes_{Z_\p} S_{Z_\p}(\Omega_{Z_\p/\k})$ is injective: this is a special case of the very general statement that a $Z$-module $M$ embeds in the product of the localisations over $\Spec Z$. We deduce from the diagram that the natural map $A\otimes_Z S_Z(\Omega) \xtwoheadrightarrow{} \gr A^e$ is an injection, hence an isomorphism as required.

We now prove (2). There is a surjective homomorphism of $(A, Z^e)$-bimodules
\begin{eqnarray*}
\phi &:& A \otimes_Z Z^e \longrightarrow A^e;\\
& & a \otimes u \longrightarrow a u.
\end{eqnarray*}
Here we view $A$ as a subalgebra of $A^e$ as explained in Remark~\ref{R:relsremark} 
and $Z^e \rightarrow A^e$ is the map described in Lemma~\ref{L:psilemma}. The kernel of $\phi$ is an $A$-linear
dependence between the ordered monomials $\delta(z_I)$ in $A^e$ but by part (1) we know that all
such dependences are trivial, whence (2). Part (3) follows by a symmetrical argument.$\hfill \qed$
\end{proofofPBW}

We now list some results which follow easily from Theorem~\ref{T:PBWenvalg}. We thank the referee for pointing out the proof of freeness in part (iv) of the following result.
\begin{Corollary}
\label{C:faithfully}
Suppose that $Z$ is regular and affine. Then the following hold:
\begin{enumerate}
\item[(i)] The natural map $Z^e \rightarrow A^e$ from Lemma~\ref{L:psilemma} is an inclusion.
\item[(ii)] If $A$ is a free $Z$-module then $A^e$ is a free (left  and right) $Z^e$-module.
\item[(iii)] If $A$ is a projective $Z$-module then $A^e$ is a projective (left and right) $Z^e$-module.
\item[(iv)] $A^e$ is a free (left and right) $A$-module, hence $A^e$ is projective and faithfully flat over $A$.
\end{enumerate}
\end{Corollary}
\begin{proof}
The PBW theorems for $Z^e$ and $A^e$ show that the map $Z^e\rightarrow A^e$ is injective on the level of associated graded algebras, proving (i).
Part (ii) follows from parts (2) and (3) of the PBW theorem, whilst (iii) is an application of Hom-tensor duality. Part (iv) requires slightly more work, and we begin by showing that $A^e$ is a countable direct sum of projective $A$-module.
As we noted many times previously, when $Z$ is regular and affine we have $\Omega$ finitely generated and projective.
Write $S_Z(\Omega) = \bigoplus_{k\geq 0} S_Z^k(\Omega)$ for the $Z$-module decomposition into symmetric powers.
Since projective modules of finite type are retracts of finite rank free modules, and since symmetric powers $S^k_Z$ preserve
retracts and free modules of finite rank, we see that $S_Z^k(\Omega)$ is projective of finite type for all $k \geq 0$.
If $S_Z^k(\Omega) \oplus Q_k = Z^{n(k)}$ for $Z$-modules $\{Q_k\mid k \geq 0\}$ and integers $\{n(k) \in \N \mid k \geq 0\}$ then we see
$A^{n(k)} \cong A \otimes_Z Z^{n(k)} \cong (A\otimes_Z S_Z^k(\Omega)) \oplus (A \otimes_Z Q_k)$ and so $A\otimes_Z S_Z^k(\Omega)$
is a projective $A$-module. It follows that the exact sequence
$$0 \rightarrow \F_{k-1} A^e \rightarrow \F_k A^e \rightarrow A \otimes_Z S_Z^k(\Omega) \rightarrow 0$$
splits for all $k \geq 0$, which implies that $A^e$ is a direct sum of projective (left) $A$-modules, hence projective. In this last deduction we have used the fact that $\F_0 A^e \cong A$ is a projective $A$-module. A symmetrical argument shows that $A^e$ is projective also as a right $A$-module.

We have actually shown that $A^e$ is a direct sum of countably many projective $A$-modules. It follows that if $I \subseteq A$ is a two sided ideal then $A^e / IA^e$ is not finitely generated as an $A$-module. In the language of \cite{Ba62} we have that $A^e$ is an $\aleph_0$-big projective $A$-module.
By Lemma~\ref{L:noetherian} we know that $A$ is noetherian so $A^e$ satisfies the hypotheses of \cite[Corollary~3.2]{Ba62} and $A^e$ is a free left $A$-module; by symmetry it is also free as a right $A$-module. Faithful flatness follows immediately.
\end{proof}

\section{Poisson primitive ideals vs. annihilators of simple Poisson $A$-modules}
\label{S:primitiveandsimples}

In this section we shall prove parts (b) and (c) of the second main theorem which relate the Poisson primitive ideals of a Poisson order to
the annihilators of simple modules. \emph{For the rest of the paper the ground field will be the complex numbers $\C$.}

\subsection{Poisson primitive ideals are annihilators}
Let $Z$ be a complex affine Poisson algebra and let $A$ be a Poisson order over $Z$.
 We write
\begin{eqnarray*}
& & \I(A) := \{\text{ideals of }A\};\\
& & \I_{\P}(A) := \{\text{Poisson ideals of }A\};\\
& & \I_l(A^e) := \{\text{left ideals of }A^e\};\\
& & \I(A^e) := \{\text{2-sided ideals of }A^e\}.
\end{eqnarray*}
We will consider extension and contraction of ideals over the inclusion of $\C$-algebras $A \subseteq A^e$
\begin{eqnarray*}
\phi &:& \I(A) \longrightarrow \I_l(A^e);\\
& & I \longmapsto A^e I;\\
\medskip
\psi &:& \I_l(A^e) \longrightarrow \I(A);\\
& & J\longmapsto J\cap A. 
\end{eqnarray*}
\begin{Lemma}\label{idealslemma}
If $Z$ is a regular affine Poisson algebra then:
\begin{enumerate}
\item[(i)]{$\psi$ is the left inverse to $\phi$;}
\item[(ii)]{$\phi : \I_{\P}(A) \rightarrow \I(A^e)$ and $\psi : \I(A^e) \rightarrow \I_{\P}(A)$.}
\end{enumerate}
\end{Lemma}
\begin{proof}
The first part follows from part (iv) of Corollary~\ref{C:faithfully} and \cite[Chapter I, \S 3, No. 5, Proposition 8]{Bo}.
For $I \in \I_{\P}(A)$ we have $\phi(I) := IZ^e = Z^e I \in \I(A^e)$, thanks to relation (\ref{mdelta}) in $A^e$.
Furthermore, when $J \in \I(A^e)$ we see that the derivation $\ad(\delta(z))$ stabilises
$\psi(J) := Z\cap J \subseteq A^e$ for all $z\in Z$. Using (\ref{mdelta}) again, the latter assertion is equivalent to saying that $\psi(J) \in \I_{\P}(A)$.
\end{proof}

We now prove part (c) of the second main theorem. 
\begin{Theorem}
\label{primisann}
If $Z$ is regular and $I\in \Prim(A)$ is a primitive ideal then there exists a simple $A^e$-module $M$ such that $$\Ann_A(M) = \P(I).$$
\end{Theorem}
\begin{proof}
Since $I$ is primitive there is a maximal left ideal $L' \subseteq A$ such that $I = \Ann_A(A/L')$.
We consider the left ideal $A^e L' \in \I_l(A^e)$ containing $A^e I$ and observe that, by Zorn's lemma
there is a maximal left ideal $L \in \I_l(A^e)$ containing $A^eL'$ and the quotient $A^e/L$ is a simple left $A^e$-module.
Since $L$ is a proper ideal of $A^e$ it follows that $L\cap A$ is a proper left ideal of $A$. 
By part (ii) of Lemma~\ref{idealslemma} we have $L' = A^e L' \cap A \subseteq L \cap A$ and so the maximality of $L'$
implies that 
\begin{eqnarray}
\label{e:primisanne1}
L' = L \cap A.
\end{eqnarray}

The annihilator
$\Ann_{A^e}(A^e/L)$ is the largest two sided ideal contained in $L$, and we claim that $\Ann_{A^e}(A^e/L) \cap A = \P(I)$.
If we can show that $$\P(I) \overset{(1)}{\subseteq} \Ann_{A}(A^e/L) \overset{(2)}{\subseteq} I$$ then the claim will follow, since we know that
$\P(I)$ is the largest Poisson ideal contained in $I$, whilst $\Ann_{A}(A^e/L)$ is a Poisson ideal by part (ii) of Lemma~\ref{idealslemma}.

Observe that if $J \subseteq I$ is any Poisson ideal of $A$ then $A^e J \subseteq A^e I \subseteq L$ and so
$A^eJ \subseteq \Ann_{A^e}(A^e/L)$ by part (ii) of Lemma~\ref{idealslemma}, since $\Ann_{A^e}(A^e/L)$ is the largest two sided ideal of $A^e$ in $L$.
By part (i) of the same lemma it follows that $J = A^e J \cap A\subseteq \Ann_{A}(A^e/L)$ and so inclusion (1) follows, taking $J = \P(I)$.

Now \eqref{e:primisanne1} and Corolary~\ref{C:alphainj} together imply that $A/L' \hookrightarrow A^e / L$ embeds as an $A$-submodule,
and it follows that $\Ann_A(A^e / L) \subseteq \Ann_A(A/L') = I$, which confirms inclusion (2). The proof is now complete.
\end{proof}

\subsection{Annihilators are Poisson rational}
In order to prove part (b) of the second main theorem we make a more detailed study of the torsion subset 
of a simple module. For a Poisson $A$-module $M$ recall that we define the torsion subset by
\begin{eqnarray}
T(M) := \{z\in Z \mid zm = 0 \text{ for some non-zero } m \in M\}.
\end{eqnarray}
The next result is one of the key steps in proving part (b) of the second main theorem. Part (ii) is rather surprising at first glance,
since in general there is no reason to expect $T(M)$ to be an ideal.
\begin{Lemma} \label{IVideal}
Let $Z$ be any complex affine Poisson algebra, $A$ a Poisson order over $Z$ and $M$ a simple Poisson $A$-module.
The following hold:
\begin{enumerate}
\item[(i)] $\Ann_A(M)$ is a prime ideal of $A$;
\item[(ii)] There exists an element $m_0 \in M$ such that $$T(M) = \{z\in Z\mid zm_0 = 0\}$$ is a prime ideal of $Z$;
\item[(iii)]$\P(T(M)) = \Ann_Z(M).$
\end{enumerate}
\end{Lemma}
\begin{proof}
If $J, K \subseteq A$ are Poisson ideals with $J, K \nsubseteq \Ann_A(M)$ then $KM$ is a nonzero Poisson 
submodule of $M$ hence equal to $M$, and $JKM = M$ similarly, hence $JK \nsubseteq \Ann_A(M)$. Now the argument
of \cite[Lemma~1.1(d)]{Go} shows that $I$ is prime. Note that the minimal primes of $A$ are Poisson
thanks to paragraphs 3.1.10 and 3.3.3 of \cite{Di}, since $A$ is noetherian. This proves (i).

We suppose $M$ is simple. Recall that the associated primes $\Ass_Z(M)$ of $M$ as a $Z$-module are those prime ideals $\p\in \Spec(Z)$ of the form
$\Ann_Z(m) := \{z \in Z\mid zm = 0\}$ for some $m \in M \setminus \{0\}$. It is well known that those ideals $\Ann_Z(m_0)$ which are
maximal in the set $\{\Ann_Z(m) \mid m\in M \setminus \{0\}\}$ are prime, and hence lie in $\Ass_Z(M)$; see \cite[Proposition~3.4]{Ei} for example.
Since $Z$ is noetherian every annihilator $\Ann_Z(m)$ is contained in a maximal annihilator and this implies
\begin{eqnarray}
\label{e:Tisp}
T(M) = \bigcup_{\p \in \Ass_Z(M)} \p.
\end{eqnarray}

Choose $m_0 \neq 0$ such that $\Ann_Z(m_0) \in \Ass_Z(M)$ is maximal amongst the annihilators.
We claim that $T(M) = \Ann_Z(m_0)$. Observe that $A^e m_0 = M$ since $m_0 \neq 0$ and $M$ is simple.
Define a filtration $M = \bigcup_{i \geq 0} \F_i M$ where $\F_i M := (\F_i A^e) m_0 = A \F_i Z^e m_0$ by Lemma~\ref{L:envgenset}.
Using relations \eqref{alghom}, \eqref{mdelta} and induction on $i$ we see that $z^{i+1} \F_i M = 0$ for every $z \in \Ann_Z(m_0)$.
Pick some $z \in \Ann_Z(m_0)$ and note that if $m \in M$ has a prime annihilator in $Z$ then $m \in \F_i M$ for some $i$, so $z^{i+1} m = 0$.
This implies $z \in \Ann_Z(m)$ by primality.
We have deduced the inclusion $\Ann_Z(m_0) \subseteq \Ann_Z(m)$, which is actually an equality by the maximality of $\Ann_Z(m_0)$. Using \eqref{e:Tisp} we conclude
\begin{eqnarray}
\label{e:Tmzisann}
T(M) = \Ann_Z(m_0)
\end{eqnarray}
as desired, proving (ii). 

Set $I := \Ann_Z(M)$. We have $I \subseteq \P(T(M))$ and we now prove that this is an equality. According to \cite[3.3.2]{Di} we have
\begin{eqnarray}
\label{e:coredef1}
\P(T(M)) = \left\{z \in T(M) \mid \begin{array}{c} H(x_1) \cdots H(x_n) z \in T(M) \\ \text{ for all } n \geq 0 \text{ and } x_1,...,x_n \in Z\end{array}\right\}.
\end{eqnarray}
According to \eqref{e:Tmzisann} this set is equal to
\begin{eqnarray}
\label{e:coredef2}
\{z \in Z \mid  (H(x_1) \cdots H(x_n) z) m_0 = 0 \text{ for all } x_1,....,x_n \in Z, n \geq 0\}
\end{eqnarray}
We claim that \eqref{e:coredef2} is equal to
\begin{eqnarray}
\label{e:coredef3}
\{z \in Z \mid z \nabla(x_1) \cdots \nabla(x_n) m_0 = 0 \text{ for all } x_1,...,x_n \in Z, n \geq 0\}
\end{eqnarray}
where $\nabla : Z \rightarrow \End_\C(M)$ is the structure map of the module $M$. To prove they are equal we define
\begin{eqnarray*}
& & I_k := \{z \in Z \mid  (H(x_1) \cdots H(x_n) z) m_0 = 0 \text{ for all } x_1,....,x_n \in Z, k \geq n \geq 0\};\\
& & J_k := \{z \in Z \mid z \nabla(x_1) \cdots \nabla(x_n) m_0 = 0 \text{ for all } x_1,...,x_n \in Z, k \geq n \geq 0\},
\end{eqnarray*}
and we show that $I_k = J_k$ for all $k \geq 0$. The case $k = 0$ is trivial and so we prove the
case $k > 0$ by induction. By part (i) of Lemma~\ref{L:extendedrels} and part (iii) of Lemma~\ref{L:categoryequivalence}
we have
\begin{eqnarray}
\label{e:coredef4}
(H(x_1) \cdots H(x_n) z) m_0 = (\ad(\nabla(x_1)) \cdots \ad(\nabla(x_n)) z) m_0.
\end{eqnarray}
The right hand side of \eqref{e:coredef4} is a sum of expressions of the form
\begin{eqnarray}
\label{e:coredef5}
\pm \nabla(x_{j_1}) \cdots \nabla(x_{j_p}) z \nabla(x_{j_p+1}) \cdots \nabla(x_{j_n}) m_0
\end{eqnarray}
where $\{j_1,...,j_n\} = \{1,...,n\}$ and $0 \leq p \leq n$, and there is a unique summand in \eqref{e:coredef4} with $p = 0$,
in which case $(j_1,...,j_n) = (n, n-1,...,1)$. If $z \nabla(x_1) \cdots \nabla(x_n) m_0 = 0$ for all $x_1,...,x_n \in Z, k \geq n \geq 0$
then it follows immediately that \eqref{e:coredef4} vanishes, whence $J_k \subseteq I_k$. Conversely, if $z \in I_k$ then $z \in J_{k-1}$
by the inductive hypothesis, and so we deduce $z \nabla(x_k) \cdots \nabla(x_1) m_0 = 0$ for all $x_1,...,x_k \in Z$ 
from our description of the summands occuring in \eqref{e:coredef4}. This shows that $I_k \subseteq J_k$.

Since \eqref{e:coredef2} is equal to $\bigcap_{k \geq 0} I_k$ and \eqref{e:coredef3} is equal to $\bigcap_{k \geq 0} J_k$ we have
proven that $\P(T(M))$ is given by \eqref{e:coredef3}. It follows that this ideal annihilates $Z^e m_0$.
By Lemma~\ref{L:envgenset} we see that $M = A^e m_0 = A Z^e m_0$. If $z \in \P(T(M))$ then since $z$ is central in $A$
we have $z M = A (z Z^e m_0) = 0$ and we have shown that $\P(T(M)) = I$. This proves (iii).
\end{proof}

We are ready to prove part (b) of the second main theorem.
\begin{Theorem}
\label{annisprim}
Let $Z$ be a complex affine Poisson algebra and $A$ a Poisson order over $Z$.
If $M$ is a simple Poisson $A$-module then $\Ann_A(M)$ is Poisson rational.
\end{Theorem}
\begin{proof}
Write $I := \Ann_A(M)$ and observe that the quotient $A/I$ is prime thanks to part (i) of Lemma~\ref{IVideal}.
The image of the torsion subset $T(M)$ in $A/I$ is denoted $\oT(M)$. Since $A/I$ is a finite module
over the central subalgebra $Z/Z \cap I$ it follows that $Q(A/I) \cong A/I \otimes_{Z/Z\cap I} Q(Z/I)$.
In other words, when considering elements $ab^{-1} \in Q(A/I)$ we may choose a representative such
that $b \in Z/Z\cap I$. Notice that $M$ is a Poisson $Z$-module and so $Z\cap I = \Ann_Z(M)$ is a Poisson ideal
of $Z$, and consequently $Z/Z\cap I$ is a Poisson algebra.

Let $ab^{-1} \in C_{\P} Q(A/I)$ (defined in \eqref{e:Pcentre}) with $b\in Z/Z\cap I$. We claim that $ab^{-1}$
has a representative such that $b \notin \oT(M)$.
To see this, suppose that $b \in \oT(M)$ and consider the ideal
$$J := \{ z\in Z/Z\cap I \mid z x ab^{-1}, \{z, x ab^{-1}\} \in A/I \text{ for all } x \in A\}.$$
It is not hard to see that:
\begin{itemize}
\item[(i)]{$J$ is a Poisson ideal of $Z/Z \cap I$ so $A J/ I$ is an $H(Z/Z\cap I)$-stable ideal of $A/I$;}
\item[(ii)]{$b^2 \in J$ so that $J \neq 0$.}
\end{itemize}
If $J \subseteq \oT(M)$ then we would have $\P(\oT(M))\neq 0$. By Lemma~\ref{L:Pandcap} this would contradict the fact that $\P(T(M)) = I \cap Z$, as demonstrated in Lemma~\ref{IVideal}.
It follows that there exists $a_1 \in A$ and $b_1 \in (Z/Z\cap I) \setminus \oT(M)$ such that $b_1 ab^{-1} = a_1$, which implies that $ab^{-1} = a_1 b_1^{-1}$
and this confirms our claim.

We proceed with $ab^{-1} \in C_{\P} Q(A/I)$ and $b \in (Z/Z\cap I) \setminus \oT(M)$. We can choose $c \in A$ and $d \in Z \setminus T(M)$ such that $c\mapsto a$, $d\mapsto b$ under
the natural homomorphism $A \rightarrow A/I$, and this allows us to define the localisation $M_d := A_d^e \otimes_{A^e} M$.
Write $\phi : A_d \rightarrow \End_\k(M_d)$ for the representation induced by the $A_d$-module structure.
Since $ab^{-1} \in C_{\P}Q(A/I)$ we have $H(Z) cd^{-1} \subseteq I$ and $\ad(A) cd^{-1} \subseteq I$, and so by relations (ii) and (iii) of a Poisson $A_d$-module $cd^{-1}$ is sent to
$\End_{A^e_d}(M_d)$ under $\phi$. Since $M$ is a simple Poisson $A$-module and $d\notin T(M)$ we have that, $M_d$ is a non-zero simple Poisson $A_d$-module
by Proposition~\ref{P:simplepoisson}. Since $M \rightarrow M_d$ is a Poisson $A$-module homomorphism it is necessarily injective.
 By Lemma~\ref{L:envgenset} we have that $A_d^e$ is a finitely generated $\C$-algebra and so we may apply Dixmier's
lemma to deduce that $\End_{A^e_d}(M_d) = \C$. It follows that $\phi(cd^{-1} - \lambda) = 0$ for some $\lambda \in \C$,
which implies $\phi(c - \lambda d) = 0$. Since $M \hookrightarrow M_d$ we deduce that $c - \lambda d \in I$
and so $a = \lambda b$ in $A/I$. Finally we deduce that $ab^{-1} = \lambda$ in $Q(A/I)$.
This shows that $I$ is rational and completes the proof.
\end{proof}

\begin{Remark}
\label{R:hypothesisremark}
\begin{enumerate}
\item[(i)] It seems credible that the hypothesis $Z$ is regular can be removed from Theorem~\ref{primisann}; see the remarks in \textsection \ref{S:relres}
for a suggested approach in some special cases.
\item[(ii)] The hypothesis that $Z$ is affine is necessary in the statement of Theorem~\ref{annisprim}, as the following example shows.
We are grateful to Sei-Qwon Oh for explaining this to us, and permitting us to reproduce it here. Let $Z = \C[[x]]$ be the ring of formal power series in one
variable, and let $M = \C((x)) =Z[x^{-1}]$ be the ring of formal Laurent series. Since $Z$ is local it is not hard to see that the $Z$-submodules of $M$ are all of the form
$M_k := \{\sum_{i \geq k} a_i x^i \mid a_k, a_{k+1},... \in \C\}$ for $k \in \mathbb{Z}$. Now equip $Z$ with the trivial Poisson structure and make $M$ a Poisson $Z$-module
by setting $H(x) = \{x, \cdot\} := x^{-1} \frac{\partial}{\partial x}$. It is not hard to see that $H(x) : M_k \rightarrow M_{k-1}$ for all $k\in \mathbb{Z}$
and so $M$ is a simple Poisson $Z$-module. It follows that $(0)$ is the annihilator of a simple Poisson $Z$-module, however it is clearly not Poisson rational since $C_{\P}(Q(Z)) = M$.

\end{enumerate}
\end{Remark}

\section{The weak Poisson Dixmier Moeglin equivalence}
\label{S:weakPDME}
Once again the ground field is $\C$. In this section we prove (a) and the equivalence of (i), (ii), (iii) from the second main theorem.

\subsection{$\Delta$-ideals in $\Delta$-algebras}
It will be convenient to work in a slightly more general context than the setting of Poisson orders over affine algebras: we do not need to assume that the derivations $H(Z)$ arise
from a Poisson structure in order to state and prove that Poisson weakly locally closed, Poisson primitive and Poisson rational ideals all coincide.
We proceed by stating all of the notations needed, which shall remain fixed throughout the current section.

Let $A$ be a finitely generated semiprime noetherian $\C$-algebra which is a finite module over some central subalgebra $Z$.
By the Artin--Tate lemma it follows immediately that $Z$ is an affine algebra. The centre of $A$ will be written $C(A)$ for the current section.
We continue to denote the primitive and prime spectra of $A$ by $\Prim(A)$ and $\Spec(A)$, endowed with their Jacobson topologies (Cf. \textsection \ref{S:sympcores}).

We fix for the entire section an arbitrary subset $\Delta \subseteq \Der_\C(A)$ such that $\Delta(Z) \subseteq Z$,
and we remark that we do not need to assume that $\Delta$ is a Lie algebra, or even a vector space in what follows.
When $I$ is any subset of $A$ we write $\Delta(I) \subseteq I$ whenever
$\delta(I)\subseteq I$ for all $\delta \in \Delta$. We say that an ideal $I$ of $A$ is a \emph{$\Delta$-ideal} if $\Delta(I) \subseteq I$.
For every ideal $I \subseteq A$ we consider the $\Delta$-core of $I$, denoted $\P_\Delta(I)$,
which is the unique maximal two-sided $\Delta$-ideal of $A$ contained in $I$. It is is easy to see that such an ideal exists and is unique
since it coincides with the sum of all $\Delta$-ideals contained in $I$.
The \emph{$\Delta$-primitive ideals} of $A$ are the ideals $$\Prim_\Delta(A) := \{\P_\Delta(I) \mid I \in \Prim(A)\}.$$

An ideal is called $\Delta$-prime if whenever $J, K$ are $\Delta$-ideals satisfying $JK \subseteq I$
we have $J \subseteq I$ or $K \subseteq I$. The $\Delta$-spectrum of $A$ is the space of all
$\Delta$-prime ideals, equipped with the Jacobson topology, denoted $\Spec_\Delta(A)$.
\begin{Lemma}\label{primebasics}
The following hold:
\begin{enumerate}
\item{If $I\in \Spec(A)$ then $\P_\Delta(I) \in \Spec(I)$;}
\item{If $I$ is a $\Delta$-ideal and $I_1,...,I_n$ are the minimal prime ideals over $I$ then $I_1,...,I_n$ are $\Delta$-ideals;}
\item{$\Spec_\Delta(A) = \{ I \in \Spec(A) \mid \Delta(I) \subseteq I\}$;}
\item{If $\{I_s \mid s\in S\}$ is any collection of ideals of $A$ then $$\bigcap_{s\in S} \P_\Delta(I_s) = \P_\Delta(\bigcap_{s\in S} I_s)$$
and 
$$\P_\Delta(I) \cap Z = \P_\Delta(I \cap Z)$$ for all $I \in \Spec(A)$. 
}
\end{enumerate}
\end{Lemma}
\begin{proof}
Part (1) and (2) are \cite[3.3.2]{Di}. It is clear that the $\Delta$-ideals in $ \Spec(A)$ all lie in $\Spec_\Delta(A)$.
Note that if $I \in \Spec_\Delta(A)$ with minimal primes $I = \bigcap_{i=1}^n I_i$ then part (2) implies that $I = I_i$ for some $i$
and so $I \in \Spec(A)$, which gives the reverse inclusion, proving (3). Part (4) is a short calculation which we leave to the reader.
\end{proof}

An ideal $I$ in $\Spec_\Delta(A)$ is called \emph{$\Delta$-locally closed} if $\{I\}$ is a locally closed subset of $\Spec_\Delta(A)$. The following lemma is
immediate from the definition.
\begin{Lemma}\label{intover}
An ideal $I\in \Spec_\Delta(A)$ is $\Delta$-locally closed if and only if $I$ is properly contained in the intersection of all $\Delta$-ideals strictly containing it. $\hfill \qed$
\end{Lemma}

Since $A$ is noetherian and $\Spec_\Delta(A)\subseteq \Spec(A)$ we see that $A/I$ is a prime noetherian ring for all
$I \in \Spec_\Delta(A)$. By Goldie's theorem the set of nonzero elements of $A/I$ satisfies the right Ore condition and so we can consider
the full ring of fractions $Q(A/I)$ which is a simple artinian ring. The set of derivations $\Delta$ acts naturally on
$A/I$, and the action extends to an action on $Q(A/I)$ by the Leibniz rule \eqref{leibniz}.
We define the $\Delta$-centre of $Q(A/I)$ to be the set
$$C_\Delta(Q(A/I)):= \{z \in C(Q(A/I)) \mid \Delta(z) = 0\}$$
and we say that an ideal $I \in \Spec_\Delta(A)$ is \emph{$\Delta$-rational} if $C_\Delta(Q(A/I)) = \C$.

\begin{Lemma}
\label{L:locimpprimimprat}
\begin{enumerate}
\item{Every $\Delta$-locally closed ideal of $A$ is $\Delta$-primitive;}
\item{Every $\Delta$-primitive ideal of $A$ is $\Delta$-rational.}
\end{enumerate}
\end{Lemma}
\begin{proof}
Let $I$ be $\Delta$-locally closed. By Lemma 9.1.2(ii) and Corollary 9.1.8(i) of \cite{MR} we know that $A$ is a Jacobson ring,
and so $I = \bigcap_{s \in S} I_s$ for some index set $S$ where each $I_s$ is a primitive ideal of $A$. From part (4) of Lemma~\ref{primebasics}
we deduce that $I = \bigcap_{s\in S} \P_\Delta(I_s)$. Now by Lemma~\ref{intover} it follows that $I = I_s$ for some $s\in S$, hence $I$ is $\Delta$-primitive.

We now prove (2), so suppose that $I = \P_\Delta(J)$ is $\Delta$-primitive, and that $J \subseteq A$ is primitive.
We shall show that $C_\Delta(Q(A/I))\hookrightarrow A/J_0A$ where $J_0 = J \cap Z$. Since $J$ is primitive Dixmier's lemma
tells us that $J_0$ is a maximal ideal of $Z$. Since $A$ is a finite module over $Z$ it follows that $A/J_0A$ is finite dimensional over $\C$,
thus $\C$ is the only subfield of $A/J_0A$. Hence once we have proven the existence of such an embedding the lemma will be complete.

After replacing $A$ by $A/I$ we can assume that $\P_\Delta(J) = 0$ and show that $C_\Delta(Q(A))\hookrightarrow A/J_0A$.
Since $A$ is prime and finite over its centre we have $Q(A) \cong A \otimes_{Z} Q(Z)$ (Cf. \textsection \ref{S:thePDMEsec}) and this isomorphism is $\Delta$-equivariant.
Now if $a \otimes z^{-1} \in C_\Delta(A\otimes_{Z} Q(Z)))$ then by (\ref{leibniz}) we have $a\otimes z^{-1} = \delta(a) \otimes \delta(z)^{-1}$ for all $\delta \in \Delta$.
Since $J$ contains no nonzero $\Delta$-ideals, we can use this observation repeatedly to find a representative $a_1 \otimes z_1^{-1}$ of $a \otimes z^{-1}$
such that $z_1 \notin J$. In other words we have $a \otimes z^{-1}\in A \otimes_Z Z_{J_0} $ where $Z_{J_0}$ denotes the localisation of $Z$
at the prime $J_0 \unlhd Z$. We have
$$A\otimes_Z Z_{J_0} / J_0Z_{J_0} \cong A \otimes_Z Z/J_0 \cong A/J_0A$$
and so there is a map
$$C_\Delta(Q(A)) \hookrightarrow A \otimes_Z Z_{J_0} \twoheadrightarrow A/J_0 A.$$
The composition is necessarily an embedding since $C_\Delta(Q(A))$ is a field.
\end{proof}

\subsection{The $\Delta$-rational ideals are the $\Delta$-primitive ideals}
Now we suppose that $S \subseteq Z$ is a multiplicative subset, so that the localisation $ZS^{-1}$ may be defined. Notice that $S$ is also
a multiplicative subset of $A$ and the ring $AS^{-1} := A \otimes_Z ZS^{-1}$ satisfies the universal property of the localisation of $A$ at $S$.
In the following we identify $ZS^{-1}$ with a subalgebra of $AS^{-1}$ and, for the sake of economy, we view $A$ and $S^{-1}$ as subsets of $AS^{-1}$.
Notice that the derivations $\Delta$ extend uniquely to a set of derivations of $AS^{-1}$ via the Leibniz rule (\ref{leibniz}).

There are natural operations which send ideals of $A$ to ideals of $AS^{-1}$ and vice versa, as follows. For each ideal $I \subseteq A$
we define the extension $I^e := I \otimes_Z ZS^{-1}$ which is an ideal of $AS^{-1}$, and for each ideal $I \subseteq AS^{-1}$ we define the contraction
$I^c := I \cap A$. The following classical fact can be found in \cite[Theorem~10.20]{GW}, for instance.
\begin{Lemma}
The following hold:
\begin{enumerate}
\item[(i)]{$I^{ce} = I$ for every ideal $I$ of $AS^{-1}$;}
\item[(ii)]{$I^{ec} = I$ for each ideal $I$ of $A$ such that $A/I$ is $S$-torsion free;}
\item[(iii)]{Every ideal in the set $\Spec_S(A) := \{I \in \Spec(A)\mid I \cap S = \emptyset\}$ is $S$-torsion free.}
\end{enumerate}
\end{Lemma}

The lemma leads directly to a crucial proposition, which is probably well known.
\begin{Proposition}\label{idealsbijection}
The following hold:
\begin{enumerate}
\item[(1)]{Extension and contraction of ideals define inverse bijections between the sets $\Spec_S(A)$
and $\Spec(AS^{-1})$;}
\item[(2)]{These bijections preserve the subsets of $\Delta$-ideals;}
\item[(3)]{When $AS^{-1}$ is countably generated these bijections also preserve the set of all primitive ideals.}
\end{enumerate}
\end{Proposition}
\begin{proof}
It is a fact, easily checked using part (i) of the previous lemma, that contraction through a central extension of rings preserves prime ideals.
For the reader's convenience, we check that extension sends $\Spec_S(A)$ to $\Spec(AS^{-1})$. Pick $P \in \Spec_S(A)$ and suppose that $IJ \subseteq P^e$.
Then $I^c J^c \subseteq P^{ec} = P$ by part (ii) and (iii) of the previous lemma, and so we may assume $I^c \subseteq P$ by primality.
Using part (i) of that same lemma once again $I = I^{ce} \subseteq P^e$ and so ideals in $\Spec_S(A)$ extend to $\Spec(AS^{-1})$ as claimed. 
Now apply all three parts of the previous lemma to deduce that extension and contraction are inverse bijections on prime ideals, proving (1) of the current Proposition.
The fact that the $\Delta$-ideals are preserved is an immediate consequence of the Leibniz rule for derivations. The statement regarding
primitive ideals requires slightly more work, as we now explain.

Suppose that $M$ is a simple $A$-module with $I = \Ann_A(M)$ satisfying $I \cap S = \emptyset$. We claim that $M S^{-1} := M \otimes_Z ZS^{-1}$ is a simple nonzero $AS^{-1}$-module.
The kernel of map $M \rightarrow M[S^{-1}]$ consists of $m \in M$ such that $sm = 0$ for some $s \in S$. If such an $m \neq 0$ exists then $M = Am$ and so $sM = 0$, meaning $s \in I \cap S$.
Since this is not the case, $MS^{-1}$ is non-zero and we conclude it is also simple over $AS^{-1}$, by essentially the same argument as we used in
the proof of Proposition~\ref{P:simplepoisson}. What is more,
the reader can easily verify that $\Ann_A(M) = A \cap \Ann_{AS^{-1}} (MS^{-1})$. This shows that every primitive ideal in $\Spec_S(A)$
is equal to $J \cap A$ for some $J \in \Prim(AS^{-1})$. 

We claim that whenever $M$ is a simple $A S^{-1}$-module there exists a simple $A$-submodule $N \subseteq M$.
To show that such a simple $A$-submodule $N\subseteq M$ exists, we observe that $AS^{-1}$ is a countably generated $\C$-algebra
hence it satisfies the endomorphism property \cite[Proposition~9.1.7]{MR}. Since $A$ is a finite module over $Z$ it follows that $AS^{-1}$ is a finite $ZS^{-1}$-module,
say $AS^{-1} = \sum_{i=1}^t ZS^{-1} a_i$ for certain elements $a_1,...,a_t \in AS^{-1}$. Therefore, for any $0 \neq m \in M$, we have $M = A m = \sum_{i=1}^t ZS^{-1} a_i m = \sum_{i=1}^t \C a_i m$.
It follows that $M$ is a finite dimensional $\C$-vector space and so by an elementary inductive argument $M$ contains a simple $A$-submodule $N \subseteq M$.
Observe that we have a map $NS^{-1} \rightarrow M$ given by $n \otimes zs^{-1} \mapsto zs^{-1}n$. The kernel must be trivial since $NS^{-1}$
is simple, and the image must be all of $M$ since $M$ is simple and the image contains $N \neq 0$, so $NS^{-1} \cong M$.
We deduce that $\Ann_{AS^{-1}}(M) \cap A = \Ann_A(N) \in \Prim(A) \cap \Spec_S(A)$. It follows that contraction of ideals
$\Spec(AS^{-1}) \rightarrow \Spec_S(A)$ restricts to a surjective map of sets $\Prim(AS^{-1}) \rightarrow \Prim(A) \cap \Spec_S(A)$.
It is actually bijective since it is the restriction of an injective map by (1), which proves (3). 
\end{proof}

Recall that we say $A$ is an essential extension of $Z$ provided every non-zero ideal of $A$ intersects $Z$ non-trivially.
\begin{Lemma}\cite[Theorem~6.3.8]{FS} \label{esslemma} 
If $A$ is a prime $\C$-algebra and a finite extension of a central subalgebra $Z$ then $A$ is an essential extension of $Z$.
\end{Lemma}

\begin{Theorem}
Every $\Delta$-rational ideal in $\Spec_\Delta(A)$ is $\Delta$-primitive.
\end{Theorem}
\begin{proof}
We suppose that $J$ is a $\Delta$-prime ideal, that $C_\Delta(Q(A/J)) = \C$ and we aim to find a primitive ideal
$I \subseteq A$ with $\P_\Delta(I) = J$. After replacing $A$ by $A/J$ and replacing $Z$ by $Z/Z\cap J$ we see that it is sufficient
to suppose that $C_\Delta(Q(A)) = \C$ and find a primitive ideal $I \subseteq A$ with $\P_\Delta(I) = (0)$. By part (3) of Lemma~\ref{primebasics} we may adopt the hypothesis that $A$ is prime.
Since $C_\Delta(Q(Z)) \hookrightarrow C_\Delta(Q(A))$ it follows that $C_\Delta(Q(Z)) = \C$.

Let $M$ be the set of a minimal non-zero $\Delta$-prime ideals of $A$. We claim that $M$ is countable. First of all notice that, since $A$ is a finite extension of $Z$
there are finitely many prime ideals of $A$ lying over each prime ideal of $Z$. For the reader's convenience we sketch the proof of this fact.
If $\p\in \Spec(Z)$ is any prime ideal then the ideal $\p A$ is not necessarily prime, but since $A$ is a noetherian ring there are finitely many prime ideals $P_1,...,P_m$ of
$A$ over $\p A$. Suppose that $Q \in \Spec(A)$ is any ideal with $Q\cap Z = \p$. Then it follows that $Q$ contains one of the minimal primes $P_1,...,P_m$. We may suppose $P_1 \subseteq Q$.
Since $A / P_1$ is an essential extension of $Z / P_1 \cap Z$ we conclude from Lemma~\ref{esslemma} that either $Q = P_1$ or the image of $Q$ in $A/P_1$ intersects $Z/ P_1 \cap Z$
non-trivially. By assumption $Q \cap Z = \p =  P_1 \cap Z$ and so the latter does not hold, hence we may conclude $Q = P_1$. This proves that there are finitely many primes of $A$ lying over $\p$.
Now in order to prove that $M$ is countable, it suffices to show that $Z$ contains only countably many
minimal nonzero $\Delta$-prime ideals. This follows from the argument given in \cite[Theorem~3.2]{BLLM}, using the fact that $C_\Delta(Q(Z)) = \C$.

Now we may enumerate $M = \{P_1, P_2, P_3,...\}$ and write $\p_i := P_i \cap Z$ for all $i =1,2,3,...$. By assumption $A$ is prime and so thanks to Lemma~\ref{esslemma} we see that $A$ is an essential extension of $Z$.
In particular $\{\p_1, \p_2, \p_3,...\}$ are all non-zero. Choose non-zero elements $\{s_1, s_2, s_3,...\}$ with $s_i \in \p_i$ and
let $S$ be the multiplicative subset of $S$ generated by $\{s_1, s_2, s_3,...\}$. Note that $AS^{-1}$ is countably generated, so we are in a position to apply every conclusion of Proposition~\ref{idealsbijection}.

Let $I \in \Spec(AS^{-1})$ be any primitive ideal. Suppose for a moment that $I$ contains some non-zero $\Delta$-prime ideal.
Then it must contain a minimal non-zero $\Delta$-prime ideal, which we may denote by $K$. It follows from Proposition~\ref{idealsbijection} that $K \cap A$ is a non-zero $\Delta$-prime ideal
which intersects $S$ trivially. This is impossible, since $S$ contains a non-zero element of every non-zero $\Delta$-prime ideal of $A$. We may conclude that $\P_\Delta(I) = (0)$.
Now apply part (4) of Lemma~\ref{primebasics} to see that $\P_\Delta(I \cap A) = \P_\Delta(I) \cap A = (0)$. Thanks to the last part of Proposition~\ref{idealsbijection} we see that $I \cap A$ is a
primitive ideal of $A$ and so we have shown that $(0)$ is $\Delta$-primitive, as required.
\end{proof}

\subsection{The $\Delta$-weakly locally closed ideals are the  $\Delta$-rational ideals}

We now go on to prove that $\Delta$-rational ideals of $A$ enjoy a property which is strictly weaker than being $\Delta$-locally closed.
We say that an ideal $I \subseteq A$ is $\Delta$-weakly locally closed if the following set is finite
$$\{J \in \Spec_\Delta(A) \mid I \subseteq J, \Ht(J) = \Ht(I) + 1\}.$$

\begin{Theorem}
The $\Delta$-rational ideals of $A$ are precisely the $\Delta$-weakly locally closed ideals.
\end{Theorem}
\begin{proof}
After replacing $A$ with $A/I$ we
may show that $(0)$ is $\Delta$-rational if and only if it is $\Delta$-weakly locally closed.
We begin by supposing $(0)$ is not $\Delta$-rational and show that it is not $\Delta$-weakly locally closed.
Recall that we identify $Q(A)$ with
$A \otimes_Z Q(Z)$, and we identify $A$ and $(Z\setminus \{0\})^{-1}$ with subsets of $Q(A)$.
Suppose that there exists some non-constant $az^{-1} \in C_\Delta (Q(A))$ with $0 \neq z \in Z$.
We consider the localisation $Z_z := Z[z^{-1}]$ and $A_z := A \otimes_Z Z_z \subseteq Q(A)$.
For all $c \in \C$ we observe that $az^{-1} - c$ is central in $A_z$ and we consider the
ideals $I_c := (az^{-1} - c)A_z$. 

We claim that $\{I_c \mid c\in \C\}$ are generically proper ideals. If $az^{-1} - c$ is invertible in
$A_z$ then for $b\in A_z$
$$b (az^{-1} - c)^{-1} (az^{-1} - c) = b = (az^{-1} - c)^{-1} (az^{-1} - c)b = (az^{-1} - c)^{-1} b (az^{-1} - c).$$
Combined with the fact that $az^{-1} - c$ is not a zero divisor we conclude that $(az^{-1} - c)^{-1}$ is central in $A_z$.
Writing $C(A_z)$ for the centre of $A_z$ we conclude that $I_c$ is proper if and only if $J_c := (az^{-1} - c) C(A_z)$ is proper.
The Artin--Tate lemma implies that $C(A_z)$ is an affine algebra and so the ideals
$\{J_c \mid c\in \C\}$ are generically proper, which confirms the claim at the beginning of this paragraph.

Since $A$ is a noetherian ring so is $A_z$ and so by \cite[4.1.11]{MR} there are
finitely many minimal prime ideals over $I_c$ each of which has height one.
If some prime ideal contains both $I_c$ and $I_{c'}$ for some $c, c' \in \C$ then it
contains $c - c'$ and so these prime ideals are all distinct. Now we have found infinitely many $\Delta$-prime
ideals of $A_z$, all of height one. Finally we apply parts (1) and (2) of Proposition~\ref{idealsbijection} to see
that $(0)$ is not $\Delta$-weakly locally closed, as desired.

Now we show that if $(0)$ is $\Delta$-rational it is $\Delta$-weakly locally closed. First of all we observe
that $(0)$ is a $\Delta$-rational ideal of $Z$, thanks to the identification $Q(A) = A \otimes_Z Q(Z)$. Thanks to \cite[Theorem~7.1]{BLLM}
we know that $(0)$ is $\Delta$-weakly locally closed in $\Spec(Z)$. Suppose that $\p_1,...,\p_l$ are the set of those
minimal non-zero prime ideals of $Z$ which are $\Delta$-stable. Since there are finitely many prime ideals of $A$ lying above
each ideal $\p_1,...,\p_l$ it suffices to show that each one of the minimal non-zero prime ideals of $A$ which is $\Delta$-stable
lies over one of $\p_1,...,\p_l$. Let $P$ be a minimal non-zero prime of $A$ which is $\Delta$-stable and observe that $P \cap Z$ is
a $\Delta$-prime ideal. We must show that $P \cap Z$ is minimal amongst non-zero primes of $Z$. If not then there exists a prime $\q$ with
$\q \subsetneq P \cap Z$. It then follows from the going down theorem \cite[13.8.14(iv)]{MR} that there is a prime $Q$ of $A$ with $Q\cap Z = \q$ and $Q \subsetneq P$;
this contradicts the minimality of $P$ and so we deduce that $P\cap Z$ is a minimal non-zero prime, as required.
\end{proof}

\section{Proof of the first theorem and some applications}
\label{S:proofofthm1}
In this section we continue to take all vector spaces over $\C$.

\subsection{Existence of the bijection}
We begin with some topological observations about the space $\Prim(A)$ when $A$ is finite over its centre.
\begin{Lemma}
\label{L:keyspeclemma}
Let $A$ be a non-commutative ring which is finite over its affine centre $Z$. Then
\begin{enumerate}
\item If $A$ is prime, then every two open subsets of $\Prim(A)$ have a nonzero intersection.
\item If $\Prim(A)$ is stratified by irreducible locally closed sets then each stratum is determined by its closure.
\end{enumerate}
\end{Lemma}
\begin{proof}
Let $\O_1, \O_2 \subseteq \Prim(A)$ be two open subsets and suppose that $\O_1 = \{ J \in \Prim(A) \mid I_1,...,I_s \nsubset J\}$
and $\O_2 = \{ J \in \Prim(A) \mid J_1,...,J_t \nsubset J\}$ for certain ideals $I_1,...,I_s,J_1,...,J_t$. Since $A$ is prime and finite over the centre $Z$ we
have that $Z\subseteq A$ is an essential extension by Lemma~\ref{esslemma}, and so we can choose nonzero element $i_k \in I_k \cap Z$ for $k=1,...,s$ and
$j_k \in J_k\cap Z$ for $k =1,...,t$. Since $A$ is prime, not one of these elements is nilpotent, and we may form the multiplicative subset $S\subseteq Z\setminus\{0\}$
which they generate. The set $\O_1 \cap \O_2 = \{J \in \Prim(A) \mid I_1,...,I_s,J_1,...,J_t \nsubset J\}$ contains the set $\{J \in \Prim(A) \mid J \cap S = 0\}$ and
according to Proposition~\ref{idealsbijection} the latter is in bijection with $\Prim(AS^{-1})$. Note that the hypotheses of that proposition are satisfied because $A$ is a finite module
over a finitely generated algebra, hence finitely generated. We know $\Prim(AS^{-1})$ is nonempty by Zorn's lemma, which proves (1).

Suppose that
\begin{eqnarray}
\label{e:somedec}
\Prim(A) = \bigsqcup_{k\in K} L_k
\end{eqnarray}
decomposes as a disjoint union of locally closed subsets and $L_1, L_2$ are two of these sets such
that $\overline L_1 = \overline L_2 = V(I)$ for some prime $I$. Then applying part (1) to the open subsets $L_1, L_2$ of $\Prim(A/I)$ we see that $L_1 \cap L_2 \neq \emptyset$,
and so $L_1 = L_2$ since the decomposition \eqref{e:somedec} is disjoint.
\end{proof}

Suppose that $Z$ is a regular complex affine Poisson algebra, so that $\Z := \Spec(Z)$ is smooth, and that the symplectic leaves of $\Z$
are locally closed in the Zariski topology. Thanks to the last part of the second main theorem we know that the PDME holds for $A$.
By the equivalence of (I) and (II) in the second main theorem we know that the closures of the symplectic cores of $\Prim(A)$ are the sets of the form
$V(\P(I)) = \{J \in \Prim(A) \mid \P(I) \subseteq J\}$. Combining parts (b), (c) and the equivalence of (ii), (iii) from the second main theorem we deduce
that the closures of the symplectic cores are precisely the sets $\{\V(M) \mid M \text{ simple Poisson } A\text{-module}\}$.
Applying Lemma~\ref{L:keyspeclemma} we see that each symplectic core $\Co(I) \subseteq \Prim(A)$ is uniquely determined
as the open core in its closure, which shows that the map from the first main theorem is a bijection, as claimed. 

\begin{Remark}
\label{R:leafalg}
The same proof works if we replace the assumption that the leaves are algebraic with the assumption that the PDME holds for $\Z$,
which is a strictly weaker hypothesis, as may be seen upon comparing \cite[Example~3.10(v)]{GL} and \cite{Go}. Another example of
a Poisson algebra $Z$ with non-algebraic symplectic leaves for which the PDME holds is the polynomial ring $\C[x_1, x_2, x_3]$
with brackets $\{x,y\} = ax, \{y,z\} = -y, \{x, y\} = 0$. According to \cite[Example~3.8]{Go1} the leaves are non-algebraic but the PDME holds \cite{BLLM}
since the GK dimension is 3.
\end{Remark}

\subsection{The bijection is a homeomorphism}
\label{S:bijishom}
Retain the hypotheses of the previous subsection. Now that we have shown that the Poisson primitive ideals of $A$ are the annihilators of
simple modules we may denote the set of of such ideals $\PPrim(A)$. Denote the bijection described in the first main theorem by $\phi : \PPrim(A) \rightarrow \Prim_\Co(A)$.
It remains to prove that $\phi$ is a homeomorphism. We must give a precise description
of the topology on each of these two spaces and observe that $\phi$ sets up a bijection between closed subsets of its domain and codomain.

The space $\Prim_\Co(A)$ is endowed with the quotient topology, meaning that for $S \subseteq \Prim(A)$ the set
$\{\Co(I) \mid I \in S\} \subseteq \Prim_\Co(A)$  is closed if and only if $\bigcup_{I \subseteq S} \Co(I)$ is closed.
Since $A$ is noetherian the closed sets of $\Prim(A)$ have finitely many irreducible components and so the closed sets of the form
$\bigcup_{I \subseteq S} \Co(I)$ can be labelled as follows: for each finite set $S = \{I_1,...,I_n\} \subseteq \PPrim(A)$
the set
\begin{eqnarray}
\label{e:closedcores}
C_S := \{\Co(I) \mid I_j \subseteq I \text{ for some } j=1,...,n\}
\end{eqnarray} is a closed subset of $\Prim_\Co(A)$
and all closed sets are obtained in this manner. This statement uses the fact that $\PPrim(A)$ are the defining ideals of the closures of the symplectic cores,
and this is a consequence of the second main theorem.

The space $\PPrim(A)$ is endowed with its Jacobson topology (see \cite[3.2.2]{Di}) and so closed subsets of $\PPrim(A)$
are constructed as follows, for each subset $T \subseteq \PPrim(A)$
there is a closed set $D_T = \{I \in \PPrim(A) \mid \bigcap_{J \in T} J \subseteq I\}$ and all closed sets occur in this way.
For $T\subseteq \PPrim(A)$ we let $S$ be the set of primes of $A$ which are minimal over $\bigcap_{J\in T} J$.
Thanks to Proposition 3.1.10 and Lemma~3.3.3 of \cite{Di} this set is finite and consists of Poisson ideals of $A$.
It is easy to see that $D_T = D_S$ and we conclude that the closed sets of both $\PPrim(A)$ and $\Prim_\Co(A)$ can be labelled
(non-uniquely) by finite subsets of $\PPrim(A)$.

It is clear that under the assumptions of our first theorem, when $J \in \PPrim(A)$ the map $\phi$ sends the closed set
$D_{\{J\}}$ to the closed set $C_{\{J\}}$. It is also not hard to see that $\phi (D_S \cup D_T) = \phi(D_S) \cup \phi(D_T)$
for finite subsets $S, T \subseteq \PPrim(A)$ and that $D_S \cup D_T = D_{S \cup T}$ and $C_{S \cup T} = C_S \cup C_T$.
It follows that $\phi D_S = C_S$ for any finite set $S \subseteq \PPrim(A)$.
We conclude that $\phi$ sends closed sets of $\PPrim(A)$ bijectively to those of $\Prim_\Co(A)$, and so $\phi$ is
a homeomorphism.

\subsection{Solvable Lie--Poisson algebras}
\label{S:solvablePLA}
As we mentioned in Remark following the first Corollary in \textsection \ref{S:sympcoresandmodules} we can combine Dixmier's theorem
and our main theorem to deduce that both $\Prim U(\g)$ and $\PPrim \C[\g^*]$ are parameterised by $\g^* / G$ in the case where $\g = \Lie(G)$
is the Lie algebra of a complex solvable algebraic group. Our purpose here is to explain this correspondence in more detail,
and show that these parameterisations are dual to one another in a sense.

First of all let $\g$ be any finite dimensional Lie algebra (over $\C$, although this remark holds over any field) and let $D := \C[\epsilon]/ (\epsilon^2)$
be the ring of double numbers. Define the double Lie algebra by $\g_D := \g \otimes_\C D$. We claim that $\C[\g^*]^e$ is naturally isomorphic to $U(\g_D)$.
Define a map $i : \g_D \rightarrow \C[\g^*]^e$ by $i(x + \epsilon y) := \delta(x) + \alpha(y)$ where $(\alpha, \delta)$ are the structure maps of $\C[\g^*]^e$
and $x,y\in \g$. By the defining relations of $\C[\g^*]^e$ it follows that $i[x,y] = i(x) i(y) - i(y) i(x)$ for every $x, y\in \g_D$ and so by the universal property of
$U(\g_D)$ there is a homomorphism $U(\g_D) \rightarrow \C[\g^*]^e$ which is surjective because the image of $i$ generates $\C[\g^*]^e$ and injective by
the PBW theorems for $U(\g_D)$ and $\C[\g^*]^e$ respectively.

When $\g = \Lie(G)$ is solvable, Dixmier's theorem states that there exists a bijection $ I_\g : \g^*/ G \rightarrow \Prim U(\g)$ which we briefly recall. When $\chi \in \g^*$
there always exists a \emph{polarisation $\p \subseteq \g$ of $\chi$}: this is a Lie subalgebra such that $\dim \p = (\dim\g + \dim \g_\chi)/2$ where $\g_\chi$
denotes the stabiliser, and $\chi[\p,\p] = 0$. The latter condition means that $\chi$ defines a one dimensional $\p$-module $\C_\chi$.
Dixmier's map is defined for $\chi \in \g^*$ by $$I_\g(\chi) := \Ann_{U(\g)} \Ind^{U(\g)}_{U(\p)} \C_\chi.$$
According to \cite[\textsection 6]{Di} it is well defined and descends to the required bijection. 

When $\g$ is solvable and algebraic it easily seen that $\g_D$ is solvable and algebraic too. Viewing $\g$ as a $G$-module we can define $G_D := G \ltimes \g$
with $1 \ltimes \g$ an abelian unipotent subgroup, and $\g_D \cong \Lie(G_D)$. We embed $\g^*$ inside $\g_D^*$ by setting $\chi(x + \epsilon y) = \chi(x)$ for $x,y\in \g$ and $\chi \in \g^*$.
It is readily seen that $\g^*$ is a $G_D$-submodule with $1\ltimes \g$ acting trivially on $\g^*$ and on $\g^*_D / \g^*$. We may view both the submodule and the quotient as $G$-modules via the identification $G_D / 1\ltimes \g \cong G$,
and the short exact sequence $\g^* \hookrightarrow \g_D^* \twoheadrightarrow \g^*_D / \g^*$ allows us to identify $\g^* / G$ naturally with a subset and a quotient set (i.e. a set of equivalence classes) of $\g^*_D / G_D$. 

Since $\epsilon \g$ is an ideal of $\g_D$ it generates an ideal of $U(\g_D)$ and we have $U(\g) \cong U(\g_D) / (\epsilon \g)$. This induces an embedding
$\Prim U(\g) \hookrightarrow \Prim U(\g_D)$. Under the isomorphism $U(\g_D) \cong \C[\g^*]^e$ the natural map $\alpha :\C[\g^*] \hookrightarrow \C[\g^*]^e$
identifies with $\C[\g^*] \cong S(\g) \cong U(\epsilon \g) \hookrightarrow U(\g_D)$ and so there is a natural map $\Prim U(\g_D) \rightarrow \PPrim \C[\g^*]$
given by $I \mapsto I \cap U(\epsilon \g)$. Now the relationship between the parameterisations of $\Prim U(\g)$ and $\PPrim \C[\g^*]$ can be explained.
\begin{Proposition}
There is a commutative diagram
\begin{eqnarray}
\begin{array}{c}
\begin{tikzpicture}[node distance=2cm, auto]
\pgfmathsetmacro{\shift}{0.3ex}
\node (A) {$\Prim U(\g)$};
\node (G) [right of= A] {$ $};
\node (B) [right of=G] {$\Prim U(\g_D)$};
\node (H) [right of= B] {$ $};
\node (C) [right of=H] {$\PPrim \C[\g^*]$};
\node (D) [below of= A] {$\g^* / G$};
\node (E) [below of=B] {$\g_D^* / G_D$};
\node (F) [below of=C] {$(\g_D^* / \g^*) / G_D \cong \g^*/ G$};

\draw[right hook->] (A) --(B) node[above,midway] {$ $};
\draw[->>] (B) --(C) node[above,midway] {$ $};
\draw[right hook->] (D) --(E) node[above,midway] {$ $};
\draw[->>] (E) --(F) node[above,midway] {$ $};

\draw[->] (A) --(D) node[right,midway] {$I_\g$};
\draw[->] (B) --(E) node[right,midway] {$I_{\g_D}$};
\draw[->] (C) --(F) node[right,midway] {$ $};
\end{tikzpicture}
\end{array}
\end{eqnarray}
where the right hand vertical arrow is constructed in our first main theorem,
and all other arrows are defined above. The composition of the horizontal maps are the constant maps respectively sending
$\Prim U(\g)$ to the annihilator of the trivial Poisson $\C[\g^*]$-module, and sending $\g^*/G$ to the zero orbit.
\end{Proposition}

\subsection{Quantum groups and open problems}
\label{S:appoftheorem}

In the introduction we proposed two applications of the first main theorem: a description of the annihilators of simple Poisson $\C[\g^*]$-modules
when $\g$ is a complex algebraic group, and also annihilators of simple modules over the classical finite $W$-algebra. Both of these
examples are Poisson algebras and so they do not use the full force of the first main theorem. We conclude by mentioning some famous examples
where $A$ is a Poisson order over a proper Poisson subalgebra $Z$, satisfying the hypotheses of the main theorem. 

Let $q$ be a variable and consider the affine $\C[q]$-algebra $A$ generated by $X_1,...,X_n$
subject to relations $X_i X_j - q X_j X_i$ for $i < j$. This is the \emph{single parameter generic quantum affine space}.
When we specialise $q \rightarrow \epsilon$ where $\epsilon$ is a primitive $\ell$th root of unity for some $\ell > 1$,
we obtain a Poisson order. To be precise, the subalgebra $Z_0$ of $A_\epsilon := A / (q - \epsilon) A$ generated by $\{X_i^\ell \mid i =1,...,n\}$
is central, known as \emph{the $\ell$-centre}. Following the observations of \textsection \ref{S:POexamples} we see that
$Z_0$ is a Poisson algebra and $A_\epsilon$ is a Poisson order over $Z_0$. There is a $(\k^\times)^n$-action on $A_\epsilon$
rescaling the generators and it is not hard to see that there are only finitely many $(\k^\times)^n$-stable Poisson prime ideals.
It follows from the results of \cite{Go} that the PDME holds for $Z_0$ and so by Remark~\ref{R:leafalg}
our first main theorem applies to $A_\epsilon$. In particular, the space $\PPrim(A_\epsilon)$ of annihilators of simple Poisson $A_\epsilon$-modules
is homeomorphic to the set $\Prim_\Co(A_\epsilon)$.

Two other natural examples to consider are the quantised enveloping algebras $U_q(\g)$ where $\g$ is a complex semisimple Lie algebra,
and their restricted Hopf duals $\O_q(G)$; see \cite{BG} for more detail. Once again, after specialising the deformation parameter $q$ to $\epsilon$ an
$\ell$th root of unity we denote one of these algebras by $A_\epsilon$. Just as for quantum affine space the $\ell$th powers of the standard generators
in either of these examples generate a central subalgebra $Z_0$. In these cases the symplectic leaves are actually locally closed
and so our main theorem applies here too.

\begin{Problem}
For all of the families of algebras discussed above:
\begin{enumerate}
\item describe the symplectic cores of $\Prim(A_\epsilon)$ explicitly by computing the fibres of the contraction $\PPrim(A_\epsilon) \rightarrow \PPrim(Z_0)$;
\item for each symplectic core $\Co(I) \subseteq \Prim(A_\epsilon)$ construct an explicit example of a simple Poisson $A$-module $M$ such that $\V(M) = \overline{\Co(I)}$.
\end{enumerate}
\end{Problem}
Of course Theorem~\ref{primisann} shows that such a module exists but the proof is non-constructive.
We hope that by constructing modules more explicitly (for example, by generators and relations) as per Problem~(2) it will become more
apparent how we could hope to deform a simple $A_\epsilon / \m A_\epsilon$-module over the closure of the symplectic core $\overline{\Co(\m)}$, where $\m \in \Prim(Z_0)$.

\end{document}